\documentclass[12]{amsart}
\usepackage{amssymb}
\usepackage{amsfonts} 
\usepackage{latexsym}   
\usepackage{amssymb}   
\usepackage{amsmath}    
\usepackage{amsbsy}
\usepackage{amsgen}
\usepackage{amsfonts}
\usepackage{xcolor}
\usepackage{array}
\usepackage{scalerel}
\usepackage{epsfig}
\usepackage{psfrag}
\usepackage[all]{xy}
\usepackage{hyperref}
\usepackage{amscd}
\usepackage{marvosym}
\usepackage{amsopn}
\usepackage{amsthm}
\usepackage{amsfonts}
\usepackage{amssymb}
\usepackage{mathrsfs}
\usepackage{hyperref}
\usepackage{ifsym}
\usepackage{breqn}

\usepackage{todonotes}
\usepackage[marginratio=1:1,height=8.5in,width=5.5in,tmargin=1.25in]{geometry}
\usepackage[normalem]{ulem}
\usepackage{enumerate,enumitem}
\usepackage{listings}
\usepackage{tikz-cd}
\usepackage{graphicx}
\usepackage{comment}

\usepackage{latexsym}

\newcommand{\mb}{\mathbb}

\newtheorem{Theorem}{Theorem}
\newtheorem{Lemma}[Theorem]{Lemma}
\newtheorem{lemma}[Theorem]{Lemma}
\newtheorem{Corollary}[Theorem]{Corollary}

\newtheorem{Proposition}[Theorem]{Proposition}

\newtheorem{Definition}{Definition}

\newcommand{\til}{\widetilde}

\newcommand{\wh}{\widehat}

\newcommand{\ra}{\longrightarrow}

\renewcommand{\{}{\lbrace}
\renewcommand{\}}{\rbrace}

\newcommand{\interior}[1]{%
	{\kern0pt#1}^{\mathrm{o}}%
}

\DeclareMathOperator{\Iso}{I}

\DeclareMathOperator{\dd}{d}
\DeclareMathOperator{\hh}{h}

\DeclareMathOperator{\Ker}{Ker}
\DeclareMathOperator{\Stab}{Stab}
\DeclareMathOperator{\dimm}{dim}
\DeclareMathOperator{\vcd}{vcd}
\DeclareMathOperator{\KK}{K}

\newcommand{\cout}[1]{}

\def\be#1\ee{\begin{align}\begin{split} #1 \end{split}\end{align}}
\def\beq#1\eeq{\begin{align*}\begin{split} #1 \end{split}\end{align*}}

\title[Quasi-isometric rigidity in cusp-decomposable manifolds]{On the quasi-isometric rigidity of chambers and walls in cusp-decomposable manifolds}

\author[Hayde\'e Contreras Peruyero]{Hayde\'e Contreras Peruyero}

\address{Instituto de Matem\'aticas,
	Universidad Nacional Aut\'onoma de M\'exico,
	Coyoac\'an,
	04510. CDMX, M\'exico}

\email{haydeeperuyero@im.unam.mx}
\date{\today}
\begin{document}
	
	\begin{abstract}
		A cusp-decomposable manifold is a manifold constructed from a finite number of complete, negatively curved, finite volume manifolds and identifying the boundaries of truncated cusps by diffeomorphisms. 
		Using properties of the electric space of the universal cover of cusp-decomposable manifolds, we show that the inclusion of walls and pieces induces quasi-isometric embeddings. We also show that isomorphisms between fundamental groups of higher graph manifolds preserve the decomposition into pieces.
	\end{abstract}

	\maketitle
	
\section{Introduction}

In the theory of 3-manifolds, the first definition of a graph manifold is due to Waldhausen in 1967 \cite{Waldhausen}. Consider the following family of manifolds which extend the ideas of Frigerio, Lafont and Sisto \cite{FLS}. 
\begin{Definition}\cite[Definition 1]{BJS}\label{HGM}
	\begin{enumerate}
		\item For every $i=1,...,n$ with $2\leq n_{i}\leq n$, let $V_{i}$ a finite volume, complete, non-compact, pinched negatively curved $n_{i}-$manifold.
		\item Denote by $B_{i}$ the compact smooth manifold with boundary obtained by truncating the cusps of $V_{i}$, i.e., by removing from $V_{i}$ a non-maximal horospherical open neighbourhood of each cusp.
		\item Take fiber bundles $Z_{i}\ra B_{i}$ with fiber $N_{i}$ a compact quotient of an aspherical nilpotent simply connected Lie group $\til{N_{i}}$ by the action of a uniform lattice $\Gamma_{i}$ of dimension $n-n_{i}$, i.e., $N_{i}$ is diffeomorphic to $\til{N_{i}}/\Gamma_{i}$.
		\item Fix a pair of diffeomorphic boundary components between different $Z_{i}$, provided one exists, and identify the paired boundary components using diffeomorphism, so that we obtain a connected $n-$manifold.  
	\end{enumerate}
	A manifold constructed in this way is called a higher dimension graph manifold. Each $Z_{i}$ is called a piece of $M$, when $\dimm(B_{i})=n$ we have $Z_{i}=B_{i}$ and we say that the piece is a pure piece. The boundary components of the pieces $Z_{i}$ are called \textit{walls} and we denote them by $W$.
\end{Definition}

These manifolds were introduced by B\'arcenas, Juan-Pineda and Su\'arez-Serrato and have been shown to satisfy the Borel Conjecture for $n\geq 6$ \cite{BJS}. This definition also includes the manifolds described by Connell and Su\'arez-Serrato \cite{CSS}.\\
\indent Tam Nguyen Phan described a family of graph manifolds called \textit{cusp decomposable manifolds} \cite{Tam}. This is a subfamily of the manifolds constructed as in Definition \ref{HGM} where the pieces are all pure pieces. This is analogous to the JSJ decomposition of 3-manifolds.  The family of {\em twisted-doubles} described by Aravinda and Farrell is also included in the cusp-decomposable manifolds \cite{Aravinda-Farrel}. 

\indent We can associate to the universal cover of a higher graph manifold (and a cusp decomposable manifold) a Bass-Serre tree \cite{Serre}. The preimage of a vertex in the Bass-Serre tree will be called a \textit{chamber}.  The preimage of an edge in the Bass-Serre tree will be called a \textit{wall}, (see Definition \ref{def-walls-chambers}).\\

\indent We say that a piece $Z$ of $M$ is a \textit{surface piece} if the base of the fiber bundle which defines $Z$ is a hyperbolic surface. \\

\indent Our two main results are:
\begin{Theorem}\label{Isomorphism-preserve-pieces}
	Let $M_{1}$,  $M_{2}$ be two higher graph manifolds, as in definition \ref{HGM} and without surface pieces. Let $\Delta_{1}$ be a subgroup of $\pi_{1}(M_1)$ that is conjugate to the fundamental group $\pi_{1}(Z_{1,j})$ of a piece $Z_{1,j}$ in $M_{1}$, and $\varphi:\pi_{1}(M_1)\longrightarrow \pi_1(M_2)$ be an isomorphism. Then $\varphi(\Delta_{1})$ is conjugate to the fundamental group $\pi_{1}(Z_{2,\sigma(j)})$ in  $\pi_1(M_2)$ of a piece $Z_{2,\sigma(j)}$ in $M_{2}$, for some permutation $\sigma$ of the boundary components of the pieces.
\end{Theorem}

To prove this Theorem, we begin by explaining the properties of the universal cover $\til{M}$ as a tree of spaces and we analyse the Bass-Serre tree $T$ associated to the decomposition into pieces. With these properties and the fact that the fundamental group of a higher graph manifold has uniformly exponential growth (Lemma \ref{exp-growth}), we show in Lemma \ref{estabi-vertex} that the stabilizer of a vertex $v$ in $T$ can not fix any other vertex. Also, in Lemma \ref{unicidad-estabilizador-paredes} we prove the analogous property for walls. In Lemma \ref{wall1-subset-nbh-wall2}, Corollary \ref{wall-adjacent-to-chamber1} and Corollary \ref{uni-chambers}, we describe certain metric properties of walls and chambers that allow us to understand when two walls or two chambers are the same. Then, we prove that the maximal nilpotent subgroups of $\pi_{1}(M)$ with Hirsch length $n-1$ are the stabilizers of walls. This characterization of the stabilizers of walls and chambers, together with the Milnor-Svar$\check{\textrm{c}}$ Lemma enable us to proof Theorem \ref{Isomorphism-preserve-pieces}, the details are found in Section \ref{section3}.  

\indent Let $(X_{1},\dd_{1})$ and $(X_{2},\dd_{2})$ be two metric spaces. A map $f: X_{1} \longrightarrow X_{2}$ is called a $(c,d)-$quasi-isometric embedding if there exist constants $c\geq 1$ and $d\geq 0$ such that for all $x,y\in X_{1}$
	$$\frac{1}{c}\dd_{1}(x,y) - d \leq \dd_{2}(f(x),f(y))\leq c\dd_{1}(x,y)+d $$
If, in addition, there exist a constant $k\geq 0$ such that every point in $X_{2}$ lies in the $k-$neighbourhood of the image of $f$, then $f$ is called a $(c,d)-$quasi-isometry and we say that $X_{1}$ and $X_{2}$ are quasi-isometric.\\

Our second main result shows that in a cusp-decomposable manifold $M$ the fundamental groups of walls and pieces are quasi-isometrically embedded in $\pi_{1}(M)$.
\begin{Theorem}\label{quasi-isometric-embeding}
	Let $M$ be a cusp decomposable manifold. Then, the inclusion of each piece and wall induces quasi-isometric embeddings of their fundamental groups into their image in $\pi_{1}(M)$.
\end{Theorem}
This result generalizes Theorem 0.16 of Frigerio, Lafont and Sisto \cite{FLS}. They mention that their strategy to prove the quasi-isometric embedding can not be generalized to the case of cusp-decomposable manifolds. In order to circumvent these difficulties we use properties of the {\it electric space}. The electric space is the space obtained by contracting the boundary of the removed horospheres in the base manifolds $B_{i}$ to a point.\\

\indent As a consequence of Lemma 3.2 of Osin \cite{Osin}, in Proposition \ref{paths-horospheres1} we prove that the length of geodesics lying on the horoboundaries of the pieces are bounded above by the lengths of paths that connect different horoboundaries. Relying on previous results by Farb \cite{Farb}, in Lemma \ref{auxiliar-goodpath} we prove that there exists a positive constant such that the length of a good path between any near two points in a wall is bounded by above by the product of this constant and the distance between the points. Also, using the properties electric spaces, in Proposition \ref{good-path} we construct a non-backtracking good path between two points in the same internal wall, which is the analogous result for cusp-decomposable manifolds to Lemma 7.8 of \cite{FLS}. This allows us to show that the inclusion of walls and chambers endowed with the path metric are quasi-isometrically embedded in the universal cover $\widetilde{M}$ of $M$. \\

The concept of quasi-isometry plays a fundamental role. One of the central questions is to understand what algebraic properties are invariant under quasi-isometries. Stallings shows that the property of splitting over a finite group  for a finitely generated group is invariant under quasi-isometries \cite{Stallings-1968, Stallings-1971}. Papasoglu showed that for a finitely presented group, admitting a splitting over a virtually infinite cyclic group is also invariant under quasi-isometries \cite{Papasoglu-2005}. Papasoglu also showed that if the vertex groups are fundamental groups of aspherical manifolds and the edge groups are smaller than the vertex groups, then the splitting is preserved under quasi-isometries \cite{Papasoglu-2007}. As a consequence of this result,  if a group $G$ is an amalgamated product of two aspherical 3-manifolds over a surface group, then any group quasi-isometric to $G$ also splits over a virtual surface group. Bowditch shows that a one-ended hyperbolic group that is not a triangle group splits over a two-ended group if and only if its Gromov boundary has local cut points. Therefore admitting such splitting is invariant under quasi-isometries for hyperbolic groups \cite{Bowditch-1998, Bowditch-1998-2, Bowditch-1999}. Another fundamental concept is the JSJ decomposition of a manifold $M$. This canonical decomposition is unique up to isotopy and its corresponds to the decomposition of the fundamental group of $M$ as a graph of groups. Kapovich and Leeb proved that this canonical decomposition is invariant under quasi-isometries and as a consequence it is a geometric invariant of the fundamental group \cite{Kapovich-Leeb}. In this paper, we contribute to the understanding of properties that are invariant under quasi-isometries for higher graph manifolds and cusp-decomposable manifolds. \\

\indent In Section \ref{section2}, we start with a review fundamental notions of Hadamard manifolds, electric spaces, nilpotent groups, and growth type of groups. In Section \ref{section3}, we begin with some metric properties of the universal cover of a higher graph manifold $M$ and we completely characterize the set of wall stabilizers. With this, and the Milnor-Svar$\check{\textrm{c}}$ Lemma we prove that an isomorphism between the fundamental groups of two higher graph manifolds preserves the decomposition into pieces. In Section \ref{section4}, we focus on cusp-decomposable manifolds and we prove Theorem
\ref{quasi-isometric-embeding}. 

\section*{Acknowledgments}

This work was supported by a doctoral CONACYT fellowship and the grant IN104819 from PAPIIT-DGAPA-UNAM. The author is very grateful to Pablo Su\'arez-Serrrato for proposing the problem and helping through the elaboration of this paper. She also wants to thank No\'e B\'arcenas, Jean-Fran\c cois Lafont and Pierre Py for their valuable comments on this work.

\section{Preliminaries}\label{section2}

\subsection{Hadamard manifolds}

For the analysis of complete, finite volume Riemannian manifolds $M$ with pinched negative sectional curvature, we follow Eberlein \cite{Eberlein}. \\
\indent Let V be a Hadamard manifold, i.e., a complete, simply connected Riemannian manifold of non-positive sectional curvature. Denote by $\dd(x , y)$ the Riemannian metric on $V$ and assume that all geodesics in $V$ have unit speed. We say that two geodesics $\alpha,\gamma$ of $V$ are \textit{asymptotes} if there exists a constant $C > 0$ such that $\dd(\alpha(t),\gamma(t))\leq C$ for all $t\geq 0$. A \textit{point at infinity} of $V$ is an equivalence class of asymptotic geodesics. The set of all points at infinity will be denoted by $V(\infty)$.\\ 
\indent Let $\Iso(V)$ be the group of isometries of $V$. Associate to each isometry $\varphi$ a displacement function $\dd_{\varphi}: p \rightarrow \dd(p,\varphi p)$. An isometry $\varphi$ is \textit{elliptic} if $\dd_{\varphi}$ has zero minimum, is \textit{hyperbolic} if $\dd_{\varphi}$ has positive minimum and is \textit{parabolic} if $\dd_{\varphi}$ has no minimum in $V$. \\
\indent Let $V$ be a finite volume Hadamard manifold with sectional curvature $K$, and $\Lambda$ be a subgroup of the isometries of $V$. We say that $\Lambda$ is a lattice of $V$ if $\Lambda$ acts freely and properly discontinuously on $V$ and the quotient manifold $M=V/ \Lambda$ has finite volume. Moreover $M$ only has finitely many ends, each end is parabolic and Riemannian collared, and $\pi_{1}(M)$ is finitely presented.\\
\indent Let $\pi: V \rightarrow M$ be the projection map and denoted by $\tilde{\gamma}$ a geodesic in $H$ that determines a point at infinity fixed by some parabolic element $\varphi$ of $\Gamma$. We say that an end $E$ of $M$ is \textit{parabolic} if there exists a divergent geodesic ray $\gamma:\left[ 0, \infty \right) \rightarrow M $ that converges to $E$ and can be expressed as $\pi \circ \tilde{\gamma}$. \\
\indent Suppose that the sectional curvature of $V$ satisfies the condition $-b\leq K \leq -a < 0$ for some positive constants $a,b$. Then the maximal almost nilpotent subgroups of $\Lambda$ are the nonidentity stability groups $\Lambda_{x}$, and each group $\Lambda_{x}$ is finitely generated. Here $x$ is a point at infinity in $V$ fixed by some parabolic element of $\Lambda$. Also each nontrivial almost nilpotent subgroup of $\Lambda$ is contained in a unique maximal almost nilpotent subgroup. \\
\indent We will recall some of the combinatorial properties of the fundamental group of such manifolds, following Farb \cite{Farb}.\\  
\indent Let $\widetilde{M}$ be the universal cover of a complete, finite volume, pinched negatively curved Riemannian manifold $M$, i.e., there exist constants $a,b>0$ such that $-b^{2}\leq \KK(M) \leq -a^{2}< 0$. In particular, $\widetilde{M}$ is a Hadamard manifold. \\
\indent Let $x\in \widetilde{M}$, $z$ is a point at infinity and $\gamma$ is the geodesic ray from $x$ to $z$. A horosphere through $x$ with center $z$ is the limit as $t\rightarrow \infty$ of the sphere in $\widetilde{M}$ with center in $\gamma(t)$ and radius $t$. Let $\dd_{S}(x,y)$ denote the path metric in a horosphere $S$. That is, if $x,y\in S$ then $\dd_{S}(x,y)$ is the infimum of the lengths of all paths in $S$ from $x$ to $y$.\\
\begin{Theorem}(Heintze-Im Hof \cite[Theorem 4.9]{Heintze-Hof})\label{distance-in-horospheres}
	Let $\gamma$ be a geodesic tangent to a horosphere $S$ in a pinched Hadamard manifold $Y$ where the sectional curvature satisfies the conditions $-b^{2}\leq \KK(M) \leq -a^{2}< 0$. If $p,q$ are the projections of $\gamma(\pm \infty)$ onto $S$, then 
	\begin{equation*}
		\frac{2}{b} \leq \dd_{S}(p,q) \leq \frac{2}{a}.
	\end{equation*}
\end{Theorem}
The following Lemma is a well known fact from Riemannian geometry. For completeness we include a proof.
\begin{lemma}\label{geodesic bi inv}
	Let $M$ be a Riemannian manifold of $\dimm k$ with $\gamma\subseteq M$ a bi-infinite geodesic and $G$ acting on $M$ isometrically. If $G\gamma=\gamma$ then $G\leq \mathbb{R}\times H$, where $H\leq O(k-1)\times \mathbb{Z}_{2}$.
\end{lemma}
\begin{proof}
	Let $\gamma: G\rightarrow M$ be an infinite geodesic in $M$, such that $G\gamma = \gamma$. Consider an orthonormal basis $\{\gamma'(0), e_{1}, ... , e_{k-1} \}$ of the tangent space at $\gamma(0)$. By parallel transport we have an orthonormal basis at every $\gamma(t)$.\\
	Let $\rho : G \rightarrow \mathbb{R}$ be a morphism defined by $g\cdot \gamma(0) = \gamma (\rho(g))$, for each $g\in G$. \\
	We will use the parallel transport along $\gamma$, with respect to the action of $G$ to define a morphism $\psi$ as follows: 
	\begin{equation*}
		\begin{array}{c c c l}
		\psi : & G & \rightarrow & O(k) \\
		& g & \mapsto & [T_{\gamma(0)}M \xrightarrow{\dd_{g}} T_{\gamma(\rho(g))}M]
		\end{array}
	\end{equation*}
	Finally, as the tangent vectors can go to $\gamma'(-t)$ or $\gamma'(t)$ and they are all unitary vectors, then the image of $\psi$ is $O(k-1)\times \mathbb{Z}_{2}$. Therefore, $G\leq \mathbb{R}\times H$, where $H\leq O(k-1)\times \mathbb{Z}_{2}$.
\end{proof}
	
\subsection{The electric space}
Let $\til{M}$ be a pinched Hadamard manifold on which $ \pi_1(M)$ acts freely, properly and discontinuous by isometries. Let $\Gamma$ be the Cayley graph associated to $\pi_{1}(M)$ and  $\wh{\Gamma}$ be the coned-off Cayley graph of $\Gamma$ with respect to the cusp subgroup $H$. Choose a $\Gamma_{i}$ invariant set of disjoint horoballs centered on the parabolic fixed points. Denote by $Z$ the resultant space of deleting the interior of all these horoballs and endow $Z$ with the path metric. Each boundary component of $Z$ is a totally geodesic horosphere, and $\Gamma$ acts freely and cocompactly by isometries on $Z$. Choosing a base point $x\in Z$ on a horosphere, we obtain a quasi-isometry of $\Gamma$ with $Z$ given by $\gamma \mapsto \gamma \cdot x$ \cite{Farb}. \\
\indent The \emph{electric space} $\wh{Z}$ is the quotient space obtained from $Z$ by identifying points which lie in the same horospherical boundary component of $Z$. The path metric $\dd_{Z}$ of $Z$ induces a  path pseudo-metric $\dd_{\wh{Z}}$ on $\wh{Z}$ as follows. Let 
\begin{equation*}
	\dd_{Y}(x,y) = \begin{cases} 
	0   , &  \textit{ if  } x,y\in S \textrm{ for some horoboundary component of } Z\\
	\dd_{Z}(x,y) , & \textrm{  otherwise.}
	\end{cases}
\end{equation*}  
Then $\dd_{\wh{Z}}(x,y)$ is equal to the infimum of $\, \sum \dd_{Y}(x_{i},x_{i+1})\, $ over all sequences of points $x=x_1,x_2,\dots,x_n=y$. Observe that $\dd_{\wh{Z}}$ locally agrees with the path metric $\dd_{\til{M}}$ outside the horospheres.\\
\indent A path $\gamma$ in $\wh{Z}$ is called an \emph{electric path}. The \emph{electric length} of an electric path $\gamma$ is denoted by $l_{\wh{Z}}(\gamma)$, it is the sum of the lengths in $Z$ of the subpaths of $\gamma$ lying outside every horosphere. An \emph{electric geodesic} from $x$ to $y$ in $\wh{Z}$ is a path from $x$ to $y$ such that $l_{\wh{Z}}(\gamma)$ is minimal. An \emph{electric P-quasi-geodesic} is a $P$-quasi-geodesic in the pseudo-path-metric space $\wh{Z}$.\\
\indent Let $P, C$ be two positive constants. The $(P,C)-$quasi isometry $f$ from $\Gamma$ to $Z$ induces a quasi-isometry $\wh{f}:\wh{\Gamma}\longleftrightarrow \wh{Z}$, defined by $\wh{f}(v)=f(v)$ for all $v\in\Gamma$. Hence, we obtain the following commutative diagram: 
\[\begin{tikzcd}
	\Gamma \arrow{r}{f} \arrow{d}{} & Z \arrow{d}{} \\
	\wh{\Gamma} \arrow{r}{\wh{f}} &  \wh{Z} 
	\end{tikzcd}
\]
\indent Observe that $\wh{f}$ is a $(2RP,C+1)-$quasi-isometry, where $R$ is such that the distance in $\til{M}$ between any two horospheres in the collection of removed horospheres is at least $R$.\\
\indent Consider a geodesic $\gamma\in M$ that does not intersect a horosphere $S$. The visual size of $S$ with respect to $\gamma$ is the diameter of the set $T=\{s\in S \textit{ } | \textit{ }	\exists \textit{ } t \textrm{ for which } \overline{\gamma(t)s \cap S} =\{s\}\}$ in the metric $\dd_{S}$. The \emph{visual size} of $S$ is the supremum of the visual size of $S$ with respect to $\gamma$ taken over all geodesics that do not intersect $S$.
\begin{Lemma}(Farb \cite[Lemma 4.4]{Farb})\label{Lemma4.4Farb}
	Horospheres in a pinched Hadamard manifold $Y$ have uniformly bounded visual size.
\end{Lemma}
\indent Let $S\subseteq \widetilde{M}$ be a horosphere and denote by $\pi_{S}$ the orthogonal projection onto $S$. Let $\gamma$ be a geodesic not intersecting $S$ and suppose another geodesic goes from $x\in \gamma$ to some point $y\in S$. Then, if $\dd_{S}(y,\pi_{S}(x))\leq C$ for some constant $C$, the visual size of  $S$ is bounded by $2/a + 2C$.
\begin{Lemma}(Farb \cite[Lemma 4.5]{Farb})\label{trackMgeodesics}
	There exist $K=K(P)$, $L=L(P)>0$ with $P>0$ given with the following property: let $\beta$ be a electric $P-$quasi-geodesic from $x$ to $y$, and let $\gamma$ be the geodesic in $\til{M}$ from $x$ to $y$. Then any subpath of $\beta$ which lies outside of  $N_{\wh{Z}}(\gamma,K)$  must have electric length at most $L$. In particular, any electric $P-$quasi-geodesic from $x$ to $y$ stays completely inside of $N_{\wh{Z}}(\gamma, K + L/2)$.
\end{Lemma}
Another important fact is that the electric space is a $\delta-$hyperbolic pseudometric space for some $\delta> 0$ (Proposition 4.6 of \cite{Farb}). The following lemmata compare the length of an electric quasi geodesic $\beta$ that penetrates a sequence of horospheres with the length of a geodesic $\alpha$ in $\til{M}$ with the same end points. We will say that $\beta$ electrically tracks $\alpha$. 
\begin{Lemma}(Farb \cite[Lemma 4.7]{Farb})\label{projectionontohoroesferes}
	Let $\beta$ be an electric $P-$quasi-geodesic that does not penetrate $S$. Then, there exists a constant $D=D(P)$ such that the projection of $\beta$ onto $S$ has a length in $S$ of at most $D\cdot l_{\wh{Z}}(\beta)$.
\end{Lemma}
We will say that an electric quasi-geodesic $\beta\in \wh{Z}$ is a \emph{quasi-geodesic without backtracking} if for each horosphere $S\in \wh{Z}$ which $\beta$ penetrates, $\beta$ never returns to $S$ after leaving $S$. The image under $\wh{f}:\wh{\Gamma}\rightarrow \wh{Z}$ of every geodesic in $\wh{\Gamma}$ is a quasi-geodesic without backtracking in $\wh{Z}$.
\begin{Lemma}(Farb \cite[Lemma 4.8]{Farb})\label{electric4.8}
	Let $\beta$ be an electric $P-$quasi-geodesic without backtracking from $x$ to $y$ and let $\gamma$ be the geodesic in  $\til{M}$ from $x$ to $y$. Then, there exists a constant $D=D(P)$ such that if any of $\beta$ or $\gamma$ penetrates $S$, then the distance in $S$ from the point of entry of this path into $S$ to its exit point is at most $D$.
\end{Lemma}
\begin{Lemma}(Farb \cite[Lemma 4.9]{Farb})\label{electric4.9}
	Let $\beta$ be an electric $P-$quasi-geodesic without backtracking from $x$ to $y$ and let $\gamma$ be the geodesic in  $\til{M}$ from $x$ to $y$. Then, there exists a constant $D=D(P)$ such that if $\beta$ and $\gamma$ penetrate some horosphere $S$, then the entry point of $\beta$ into $S$ is at distance $D$ in $S$ from the entry point of $\gamma$ into $S$; similarly for the exit points.
\end{Lemma}
\indent Let $G$ be a finitely generated group and let $\Gamma$ be its Cayley graph. Let $\{H_{1},H_{2},...,H_{r}\}$ be a finite set of finitely generated subgroups of $G$. The coned-off Cayley graph of $G$ with respect to $\{H_{1},H_{2},...,H_{r}\}$ denoted by $\wh{\Gamma}=\wh{\Gamma}(\{H_{1},H_{2},...,H_{r}\})$ is formed as follows: for each $gH_{i}$, $1\leq i\leq r$, add a vertex $v(gH_{i})$ to $\Gamma$ and add an edge $e(gH_{i})$ of length $1/2$ from each $gh_{i}\in gH_{i}$ to the vertex $v(gH_{i})$.\\
\indent Following Farb \cite{Farb} We say that the group $G$ is hyperbolic relative to $\{H_{1},H_{2},...,H_{r}\}$ if the coned-off Cayley graph $\wh{\Gamma}$ of $G$ with respect to $\{H_{1},H_{2},...,H_{r}\}$ is a Gromov $\delta-$hyperbolic space for some $\delta$.
\begin{Theorem}(Farb \cite[Theorem 5.1]{Farb})
	Let $M$ be a complete, noncompact, finite volume Riemannian $n$-manifold with pinched negatively sectional curvature and denote by $\Gamma$ its fundamental group. Let $\{H_{1},H_{2},...,H_{r}\}$ be the cusp subgroups of $\Gamma$, then $\Gamma$ is hyperbolic relative to the set $\{H_{1},H_{2},...,H_{r}\}$ of cusp subgroups.
\end{Theorem}

\subsection{Nilpotent groups} 
A group $G$ is \textrm{solvable} if it has an abelian series, i.e., a series  $1=G_{0}\triangleleft G_{1}\triangleleft ...\triangleleft G_{n}=G$ in which each factor $G_{i+1}/G_{}i$ is abelian. A group $G$ is called nilpotent if it has a central series, i.e., a normal series $1=G_{0}\leq G_{1}\leq ...\leq G_{n}=G$ such that $G_{i+1}/G_{i}$ is contained in the center  of $G/G_{i}$ for all $i$. A nilpotent group is solvable. A group $G$ is said to be \textrm{polycyclic} if it has a cyclic series, i.e., if it has a series with cyclic factors. Polycyclic groups are solvable. Observe that if $G$ is nilpotent then G is polycyclic. In a polycyclic group $G$, the number of infinite factors in a cyclic series is independent of the series and hence it is an invariant of $G$, known as the Hirsch length \cite[5.4.13]{Robinson}. We will denote by $\hh(G)$ the Hirsch length of $G$, i.e., $\hh(G) := \# \{\textrm{ }i\textrm{ }|\textrm{ }  G_{i}/G_{i+1} \cong \mb{Z}\}$.\\
\indent We will use the following results about the Hirsch length in our proofs below.
\begin{Theorem}(Robinson \cite{Robinson})\label{Hirsch-suma}
	For a short exact sequence of polycyclic groups 
	$$1 \ra H \ra G \ra G/H \ra 1$$ 
	the following is true
	\begin{equation*}
		\hh(G)=\hh(H)+\hh(G/H).
	\end{equation*}
\end{Theorem}
\begin{Lemma}(Robinson \cite{Robinson})\label{Hirsch-subgrupos}
	If $H\leq G$ are polycyclic groups, then $\hh(H)\leq \hh(G)$.
\end{Lemma}
Let $G$ be a group with torsion free, finite index subgroups, then all such subgroups have the same cohomological dimension. This common dimension is called the virtual cohomological dimension of $G$ and it is denoted by $\vcd (G)$. 
\begin{Theorem}(Brown \cite{Brown})\label{Hirsch-vcd}
	If $G$ is a polycyclic group, then $\hh(G)=\vcd(G)$, where $\vcd$ denotes the cohomological virtual dimension. Let $M$ be closed aspherical manifold of dimension $n$ and $G=\pi_{1}(M)$. Then $\vcd(G)=n$.
\end{Theorem}
As a consequence of this Theorem, we have that $\hh(M)=n$ when $M$ is a infranilmanifold. 
\begin{Proposition}(Serre \cite[Proposition 27]{Serre})\label{Nilpotent-groups-serre}
	Let $G$ be a finitely generated nilpotent group acting on a tree $X$. Then only the following mutually exclusive cases are possible:
	\begin{enumerate}
		\item $G$ has a fixed point.
		\item There is a straight path $T$ stable under $G$ on which $G$ acts by translations by means of a non-trivial homomorphism $G\longrightarrow \mathbb{Z}$.
	\end{enumerate}
\end{Proposition}

\subsection{Injections of fundamental groups}
Denote by $c_{i}$ the number of boundary components of each piece $Z_{i}$, i.e., $c_{i}=\#\pi_{0}(\partial Z_{i})$.  We will denote the set of disjoint boundary components of $Z_{i},Z_{j}$ as follows:
\begin{align*}
	\partial Z_{i}=\coprod\limits_{\alpha=1}^{c_{i}} W_{\alpha}^{i}\\  
	\partial Z_{j}=\coprod\limits_{\alpha=1}^{c_{j}} W_{\alpha}^{j}.
\end{align*}
\indent We denote the gluing diffeomorphism between the pieces $Z_{i},Z_{j}$ by $\dd_{ij,k}$, where the index $k$ denotes the $k-$th wall of the piece $Z_{i}$, that is, $\dd_{ij,k}:W_{k}^{i}\rightarrow W_{\dd(k)}^{j}$. We sometimes omit the index $k$ when it is clear from context. This diffeomorphism induces an isomorphism between the fundamental groups,
\begin{equation*}
	\dd_{ij_{\star}}:\pi_{1}(W^{i})\rightarrow \pi_{1}(W^{j}).
\end{equation*}
\indent An end of $V_{i}$ will be denoted by $E_{i,j}$.\\
\indent  A well known result from Eberlein \cite{Eberlein} tell us the following. Each $V_{i}$ can be retracted to $B_{i}=V_{i}-\cup _{i=1}^{k} S_{i}$. Here $S_{i}$ are non-maximal horospheres removed from $V_{i}$. Then, for each end $E_{i}$ of $V_{i}$ the map  $\pi_{1}(E_{i})\rightarrow \pi_{1}(B_{i})$ is injective. We have the following commutative diagram:
	\[\begin{tikzcd}
	0 \arrow{r}{} & \pi_{1}(N_{i}) \arrow{r}{} \arrow[hook]{d}{} & \pi_{1}(W^{ij}) \arrow{r}{} \arrow[hook]{d}{} & \pi_{1}(E_{i}) \arrow{r}{} \arrow[hook]{d}{}& 0 \\
	0 \arrow{r}{} & \pi_{1}(N_{i}) \arrow{r}{}  & \pi_{1}(Z_{i}) \arrow{r}{} & \pi_{1}(B_{i}) \arrow{r}{} & 0 
	\end{tikzcd}
	\]
\indent Therefore the map $\pi_{1}(W^{ij})\hookrightarrow \pi_{1}(Z_{i})$ is injective. As a consequence of this, the fundamental group of a higher graph manifold $M$ constructed as in Definition \ref{HGM} is isomorphic to the fundamental group of a graph of groups $\mathcal{G}_{M}$. The vertex groups are the fundamental groups of the pieces $\pi_{1}(Z_{i})$ and the edge groups are $\pi_{1}(W^{ij})$ (see \cite{HP,FLS}).

\subsection{Uniformly exponential growth of the fundamental group} Let $\Gamma$ be a group finitely generated by the finite set of generators $S$. The word length $l_{S}(\gamma)$ of an element $\gamma\in \Gamma$ is defined as the smallest integer $n$ for which there exist $s_{1},s_{2},...,s_{n}\in S\cup S^{-1}$ such that $\gamma=s_{1}s_{2}\cdots s_{n}$.  The \textbf{word metric} $\dd_{S}(\gamma_{1},\gamma_{2})$ is defined as the length $l_{S}(\gamma_{1}^{-1}\gamma_{2})$. With this metric the group $\Gamma$ is a metric space.
\begin{Proposition} (Milnor-Svar$\check{\textrm{c}}$ Lemma, \cite{Harpe})\label{lemaMS} 
	Let $X$ be a proper, geodesic metric space. Let $\Gamma$ be a group acting by isometries from the left on $X$. Assume that the action is proper and the quotient $\Gamma\backslash X$ is compact. Then the group $\Gamma$ is finitely generated and quasi-isometric to $X$. Moreover, for every point $x_{0}\in X$, the map $\Gamma\rightarrow X$ given by $\gamma\mapsto \gamma x_{0}$, is a quasi-isometry.
\end{Proposition}
A useful consequence of this Lemma is that if $M$ is a compact Riemannian manifold with Riemannian universal covering $\widetilde{M}$, then the fundamental group of $M$ is quasi-isometric to $\widetilde{M}$. \\
\indent The \textbf{growth function} $\beta(\Gamma,S;k)$ is the number of elements $\gamma\in \Gamma$ such that $l_{S}(\gamma)\leq k$. The growth type of the pair $(\Gamma,S)$ is classified as follows:
\begin{enumerate}
	\item The group $\Gamma$ is of exponential growth, if there exist constants $A,B> 0$ such that for $n\geq 0$ the growth function satisfies $\beta(\Gamma, S; n) \geq A \exp ^{Bn}$. 
	\item The group $\Gamma$ has polynomial growth, if there exist constants $d,c$ such that for $n\geq 0$ the growth function satisfies $\beta(\Gamma, S; n) \leq c n ^{d}$.
	\item The group $\Gamma$ is of intermediate growth, if it is neither of exponential nor of polynomial growth.
\end{enumerate}
\noindent The growth type of a finitely generated group is a quasi-isometry invariant, i.e., quasi-isometric finitely generated groups have the same growth type \cite[Corollary 6.2.6]{Loh}. If $\Gamma$ is a finitely generated nilpotent group, then $\Gamma$ has polynomial growth \cite[Theorem 6.3.6]{Loh}. \\
\noindent The type of exponential growth of the pair $(\Gamma,S)$ is $\omega(\Gamma,S)= \underset{k\rightarrow \infty}{\lim}\sqrt[k]{\beta(\Gamma,S;k)}$. Denote by $\omega(\Gamma)=\inf \{\omega(\Gamma,S): S \ \text{ is a finite generating set of } \Gamma \}$. We have the following classification:
\begin{enumerate}
	\item The group $\Gamma$ has exponential growth if $\omega(\Gamma,S)>1$.
	\item The group $\Gamma$ has subexponential growth if $\omega(\Gamma,S)=1$.
	\item The group $\Gamma$ has uniformly exponential growth if $\omega(\Gamma)>1$.
\end{enumerate}  

\indent A well known result of de la Harpe and Bucher \cite{Harpe-Bucher} states that if $C$ is a subgroup of two finitely generated groups $A,B$ and they satisfy the condition $([A:C]-1)([B:C]-1)\geq 2$ then the free product with amalgamation $A\ast_{C} B$ has uniformly exponential growth. Using this we obtain the following result (see \cite{HP}).
\begin{Lemma}\label{exp-growth} The fundamental group of a higher graph manifold $M$ as in Definition \ref{HGM} has uniformly exponential growth.
\end{Lemma}
\begin{proof}
	The proof is by induction on the pieces. Let $Z_{i},Z_{j}$ be two adjacent pieces and $W_{i,j}$ be their common wall. We know that the map $\pi_{1}(W_{i,j})\hookrightarrow \pi_{1}(Z_{i})$ is injective and the same for the map $\pi_{1}(W_{i,j})\hookrightarrow \pi_{1}(Z_{j})$. Then by the main Theorem of \cite{Harpe-Bucher} the free product with amalgamation $\pi_{1}(Z_{i})\ast_{\pi_{1}(W_{i,j})}\pi_{1}(Z_{j})$ has uniformly exponential growth. \\
	\indent Let $Z_{k}$ be another piece adjacent to $Z_{j}$ and $W_{j,k}$ their common wall. Observe that $\pi_{1}(W_{k})$ can be seen as a subgroup of $\pi_{1}(Z_{i})\ast_{\pi_{1}(W_{i,j})}\pi_{1}(Z_{j})$ so we can again consider the free product with amalgamation, apply the same result of \cite{Harpe-Bucher} and conclude the proof.  
\end{proof}
	
\section{Isomorphisms preserve pieces}\label{section3}
The objective of this section is to prove Theorem \ref{Isomorphism-preserve-pieces}. Throughout this section we will assume that $M$ is a higher graph manifold as in definition \ref{HGM}, and $\widetilde{M}$ is its universal covering.\\
\indent Let $M$ be a manifold constructed as in definition \ref{HGM}. We will now describe the universal cover of a higher graph manifold $M$.
\indent  Let $X$ be a graph with vertices $V(X)$ and oriented edges $E(X)$. Consider a group $G$ which acts on $X$. An \emph{inversion} is a pair consisting of an element $g\in G$ and an edge $x\in E(X)$ such that $gx=\bar{x}$, here $\bar{x}$ denotes the reverse orientation on $x$. If there is no such pair we say that $G$ acts without inversion. Remember that if $\mathcal{G}$ is a graph of groups, the fundamental group $\pi_{1}(\mathcal{G})$ acts without inversion on the Bass-Serre tree $T$ associated to $\mathcal{G}$ \cite[Section 5.4]{Serre}.\\ 
\indent Suppose that $(\widetilde{M},\phi,T)$ is a tree of spaces, where $\widetilde{M}$ is the universal cover of $M$ and $T$ is the Bass-Serre tree associated to the decomposition of $M$. 
\begin{Definition}\label{def-walls-chambers}
	\begin{enumerate}
		\item \label {def-wall} A wall of $\widetilde{M}$ is the closure of the pre-image under $\phi$ of the interior of an edge of $T$. We will denote by $\dd_{W}$ the path metric induced on $W$ by the restriction to $W$ of the Riemannian structure of $\widetilde{M}$.
		\item \label{def-chamber} A chamber $C\subseteq \widetilde{M}$ is the pre-image under $\phi$ of a vertex of $T$. We will denote by $\dd_{C}$ the path metric induced on $C$ by the restriction of the Riemannian structure of $\widetilde{M}$. 
	\end{enumerate}
\end{Definition}

We say that two chambers are adjacent if their corresponding vertices in $T$ are joined by an edge. A wall $W$ is adjacent to a chamber $C$ if $W\cap C\neq \emptyset$. If $W$ is a wall, then $W$ is adjacent to a chamber $C$ if and only if the vertex corresponding to $C$ is the end point of the edge corresponding to $W$.\\
\indent Every boundary component of $Z_{i}$ is $\pi_{1}-$injective in $Z_{i}$. Hence every piece and every boundary component of a piece is $\pi_{1}-$injective in $M$. In combination with the construction of a Bass-Serre tree of spaces we obtain the following result.
\begin{Corollary}\label{BStree}
	Let $M$ be a higher graph manifolds as in definition \ref{HGM}. If $\widetilde{M}$ is it universal cover and $(\til{M},\phi,T)$ is a tree of spaces, we have the following:
	\begin{enumerate}
		\item If $C$ is a chamber of $\til{M}$, then $C$ is homeomorphic to $\til{Z_{i}}\cong \til{B_{i}}\times\til{N_{i}}$. Here $\til{Z_{i}},\til{B_{i}},\til{N_{i}}$ are the universal covers of $Z_{i}$, $B_{i}$ and $N_{i}$, in that order.
		\item If $W$ is a wall of $\til{M}$, then $W$ is homeomorphic to $\partial\til{B_{i}}\times\til{N_{i}}$ or $\partial \til{B_{j}}\times\til{N_{j}}$.
	\end{enumerate}
\end{Corollary}
Let $(\til{M},\phi, T)$ be the tree of spaces given by Corollary \ref{BStree}. The tree $T$ is called the Bass-Serre tree of $\pi_{1}(M)$ with respect to the isomorphism $\pi_{1}(M)\cong \pi_{1}(\mathcal{G}_{M})$. Let $V(T)$ denote the vertices of $T$ and $E(T)$ denote its edges. The action of $\pi_{1}(M)$ on $\til{M}$ induces an action of $\pi_{1}(M)$ on $T$. The fundamental group of a piece of $M$ coincides with the stabilizer of a vertex of $T$, and the fundamental group of a wall $W_{ij}$ corresponds to the stabilizer of an edge of $T$.
\begin{lemma} \label{estabi-vertex}
	Let $M$ be a higher graph manifolds as in definition \ref{HGM} and $T$ the Bass-Serre tree associated to the decomposition of $M$. For every vertex $v\in V(T)$ and every edge $e\in E(T)$ we denote by $G_{v},G_{e}$ their stabilizers in $\pi_{1}(M)$, in that order. If $v$ is a vertex in $T$, then $v$ is the unique vertex which is fixed by $G_{v}$. 
\end{lemma}
\begin{proof}
	Suppose that $G_{v}$ fixes another vertex $v'\neq v$, then $G_{v}$ fixes an edge $e$ that contains $v$. This implies that $G_{v}$ is contained in the stabilizer of the edge $e$. By Corollary 3.3 of \cite{Eberlein}, the stabilizers of edges are virtually nilpotent groups. By a well known result from Gromov we know that virtually nilpotent groups have polynomial growth \cite{Gromov}.\\
	As a consequence of Lemma \ref{exp-growth}, the stabilizer of a vertex has uniformly exponential growth. Therefore, the stabilizer of an edge can not be contained in the stabilizer of a vertex. We conclude that $G_{v}$ only fixes $v$. 
\end{proof}
\begin{lemma}\label{normalizar-vertex}
	Let $M$ be a higher graph manifolds as in definition \ref{HGM}. Let $Z_{1},Z_{2}$ be two pieces of $M$. Let $G_{i}<\pi_{1}(M)$ be groups conjugate to the fundamental groups  of $Z_{i}$ for $i=1,2$. Then:	\begin{enumerate}
		\item The normalizer of $G_{1}$ in $\pi_{1}(M)$ is equal to $G_{1}$.
		\item If $G_{1}$ is conjugate to $G_{2}$ in $\pi_{1}(M)$ then $Z_{1}$ is isometric to  $Z_{2}$.
	\end{enumerate}
\end{lemma}
\begin{proof}
	We will consider the action of $\pi_{1}(M)$ on the Bass-Serre tree $T$.
	\begin{enumerate}
		\item Let $v_{1}\in V(T)$, by Lemma \ref{estabi-vertex}, $v_{1}$ is the only vertex fixed by $G_{v_{1}}$ and therefore fixed by $G_{1}$. Let $g\in \pi_{1}(M)$ and suppose that $g$ normalizes $G_{1}$. So, for all $x\in G_{1}$ there exists $y\in G_{1}$ such that $gx=yg$. This implies $gxv_{1}=ygv_{1}$, because  $xv_{1}=v_{1}$, thus $gv_{1}=ygv_{1}$. Therefore, $G_{1}$ fixes $gv_{1}$ and as $v_{1}$ is the only vertex fixed by $G_{1}$, then $gv_{1}=v_{1}$. Hence, $g\in G_{1}$.
		\item Suppose there exists $g$ such that $gG_{1}g^{-1}=G_{2}$ and let $v_{1},v_{2}$ be the vertices which are fixed by $G_{1}$ and $G_{2}$ in that order. As $G_{1}$ is conjugate to $G_{2}$, $G_{1}$ fixes $v_{2}$ and $g(v_{1})$, hence $v_{2}=g(v_{1})$. Therefore the covering automorphism $g:\tilde{M}\longrightarrow \tilde{M}$ sends the chamber covering $Z_{1}$ to the chamber covering $Z_{2}$, from which $Z_{1}$ is isometric to $ Z_{2}$. \qedhere
	\end{enumerate}
\end{proof}
Remember that the graph manifold $M$ is formed by a finite union of pieces $Z_{i}$ and that each one is defined by a fiber bundle with base $B_{i}$ and fiber $N_{i}$. We will say that an element $g\in\pi_{1}(M)$ corresponds to the fiber direction of $\pi_{1}(Z_{i})$ if $g$ for $\rho: \pi_{1}(M)\ra \pi_{1}(B_{i})$ the morphism induced by the projection $M\rightarrow B_{i}$, we have that $\rho(g)=e \in \pi_1(B_i)$. That is to say, it fixes the direction associated to the base of a piece of $Z_{i}$.
\begin{lemma}\label{estabilizador-paredes} Let $M$ be a higher graph manifolds as in definition \ref{HGM} and $T$ the Bass-Serre tree associated to the decomposition of $M$. Let $W_{1},W_{2}$ be two distinct walls of $\til{M}$ and let $v\in V(T)$ be a vertex so that every path which connects $W_{1}$ and $W_{2}$ intersects the chamber corresponding to $v$. If $g\in \pi_{1}(M)$ is such that  $g(W_{i})=W_{i}$ for $i=1,2$, then $g$ is an element that corresponds to the fiber direction in $\pi_{1}(Z_{i})$, here $Z_{i}$ is the piece of $M$ that corresponds to the vertex $v$ in $T$.
\end{lemma}
\begin{proof}
	Let $\til{Z_{i}} \subset \til{M}$ be the chamber associated to $v$ and denote the piece of $M$ corresponding to $\til{Z_{i}}$ by $Z_{i}$. Let $Z_{i}^{1},Z_{i}^{2}$ be the boundary components of $\partial (\til{Z_{i}})$ such that $g(\til{Z_{i}})=\til{Z_{i}}$, $g(Z_{i}^{1})=Z_{i}^{1}$ and $g(Z_{i}^{2})=Z_{i}^{2}$. Thus we have that $g\in G_{v}$. \\
	Recall that $G_{v}$ is identified with the fundamental group of the piece corresponding to $v$, i.e., with $\pi_{1}(Z_{i})$, and that $\til{Z_{i}}\cong \til{B_{i}}\times \til{N_{i}}$. Let $\rho: \pi_{1}(M)\ra \pi_{1}(B_{i})$ be the morphism induced by the projection $M\rightarrow B_{i}$. The boundary components $\til{Z_{i}}$ are in bijection with the boundary components of $\til{B_{i}}$. \\
	Let $\gamma$ be a bi-infinite geodesic that is completely contained in $B_{i}$. By definition the stabilizer of $\gamma$ satisfies  $\Stab(\gamma)=\gamma$, then by Lemma \ref{geodesic bi inv}, $\Stab(\gamma)\leq \mathbb{R}\times H$. If $\rho(g)\gamma=\gamma$ then there exists $p\in \gamma$ such that $\rho(g)p=p$. This implies $|\rho(g)|<\infty$ and therefore $\rho(g)=\textrm{Id}$. We conclude that $g\in\Ker(\rho)$.
\end{proof}
\begin{lemma} \label{unicidad-estabilizador-paredes}
	Let $M$ be a higher graph manifolds as in definition \ref{HGM} and $T$ the Bass-Serre tree associated to the decomposition of $M$. Let $W$ be a wall in $\til{M}$ and $H$ its stabilizer in $\pi_{1}(M)$. Then $W$ is the unique wall which is stabilized by $H$.
\end{lemma}
\begin{proof}
	First, we need to note that the Hirsch length of $H$ is $\hh(H)=n-1$, because $H$ is aspherical, here $n=\dimm M$. Suppose that $H$ stabilizes another wall $W'\neq W$. By Lemma \ref{estabilizador-paredes}, $H$ is contained in $\pi_{1}(N_{v})$, for $v$ a vertex of $T$ and $N_{v}$ the corresponding fiber of the piece associated to $v$. 
	As $N_{v}$ is an aspherical manifold, using Theorem \ref{Hirsch-vcd} we have that the Hirsch length of $\pi_{1}(N_{v})$ is equal to its virtual cohomological dimension. Now, by Theorem \ref{Hirsch-suma} we have that $\hh(\pi_{1}(N_{v}))\leq n-2$. This gives us a contradiction. We conclude that $W$ is the unique wall which is stabilized by $H$.
\end{proof}
\begin{Lemma}
	Let $M$ be a higher graph manifolds as in definition \ref{HGM}. Let $Z_{1},Z_{2}$ be two pieces of $M$ and $W_{i}$ be a component of $\partial Z_{i}$, for $i=1,2$. Let $H_{i}<\pi_{1}(M)$ be groups conjugate to $\pi_{1}(W_{i})$. Then:
	\begin{enumerate}
		\item The normalizer of $H_{1}$ in $\pi_{1}(M)$ is equal to $H_{1}$.
		\item If $H_{1}$ is conjugate to $H_{2}$ in $\pi_{1}(M)$ then $W_{1}$ is isometric to $W_{2}$ in $M$.
	\end{enumerate}
\end{Lemma}
\begin{proof}
	Both proofs are analogous to those of Lemma \ref{normalizar-vertex}.
\end{proof}
We will now present some metric properties of $\til{M}$. We will denote by $\dd$ the distance associated to the Riemannian structure of $\til{M}$. For every $r>0$ and $X\subseteq \til{M}$, we will denote by $N_{r}(X)\subseteq\til{M}$ the $r-$neighborhood of $X$, with respect to the Riemannian metric $\dd$ of $\til{M}$. Recall that $\dd_{C}$ denotes the path metric on a chamber $C\in \til{M}$ and it is the induced metric on $C$ defined by the restriction of the Riemannian structure of $\til{M}$. 
\begin{Lemma}\label{distance-in-chambers}
	Let $M$ be a higher graph manifolds as in definition \ref{HGM}. If $C$ is a chamber of $\til{M}$, then there exists a function $g:\mathbb{R}^{+}\longrightarrow \mathbb{R}^{+}$ such that $g(t)\rightarrow \infty$ as $t\rightarrow +\infty$ and $\dd(x,y)\geq g(\dd_{C}(x,y))$ for each $x,y\in C$.
\end{Lemma}
\begin{proof}
	As $\dd$ and $\dd_{C}$ induce the same topology in $C$, it is enough to prove the result for a fixed point $x$. \\
	Let $\{y_{i}\}$ be a sequence of points in $C$ such that $\dd_{C}( x, y_{i})\ra \infty$. Now, as $\til{M}$ is proper, if we suppose that $\dd(x,y_{i})$ is bounded, then passing to a sub sequence if necessary, there exists $y\in \til{M}$ such that $\displaystyle\lim_{i\longrightarrow \infty } y_{i}=y$. We know that $C$ is closed in $\til{M}$, and $y$ is necessarily in $C$. This contradicts $d_{C}(x,y_{i})\longrightarrow \infty$.
\end{proof}
The proof of the following Lemma follows the same line of argument of Lemma 2.19 \cite{FLS}.
\begin{Lemma}\label{wall1-subset-nbh-wall2}
	Let $M$ be a higher graph manifolds as in definition \ref{HGM}. Let $W_{1},W_{2}$ be two walls of $\til{M}$, and suppose that there exists $r> 0$ such that $W_{1}\subset N_{r}(W_{2})$, then $W_{1}=W_{2}$. 
\end{Lemma}
\begin{proof}
	Consider the realization of  $\til{M}$ as tree of spaces, we can then reduce to the case when $W_{1}$ and $W_{2}$ are walls adjacent to a given chamber $C$. By the previous lemma, we can assume that $W_{1}$ is contained in an $r-$neighborhood of $W_{2}$ with respect to the metric of $C$. \\
	Let $Z_{i}$ be the corresponding piece to the chamber $C$ and let $B_{i},N_{i}$ be its base and fiber, in that order. We known that $C\cong \til{B_{i}}\times \til{N_{i}}$. As $M_{i}$ is a manifold that admits a negative curvature metric from which a finite number of non-maximal horospheres were removed, we have that by construction  $W_{1},W_{2}$ are projected onto two horospheres $O_{1},O_{2}$ and $O_{1}$ is contained in the $r-$neighborhood of $O_{2}$ with respect to the metric in $\til{B_{i}}$. However, the metric $\til{B}_{i}$ is bounded above by the locally symmetric metric, this forces $O_{1}=O_{2}$. Therefore, $W_{1}=W_{2}$. 
\end{proof}
As a consequence of this lemma we obtain the following two results. 
\begin{Corollary}\label{wall-adjacent-to-chamber1}
	Let $M$ be a higher graph manifolds as in definition \ref{HGM}. Let $W$ be a wall of $\til{M}$ and let $C$ be a chamber of $\til{M}$. If $W\subset N_{r}(C)$ for some $r\geq 0$, then $W$ is a wall adjacent to $C$.
\end{Corollary}
\begin{proof}
	Note that $W$ is contained in the $r-$neighbourhood of an adjacent wall to a chamber $C$.  Using Lemma \ref{wall1-subset-nbh-wall2}, we obtain that $W$ is adjacent to $C$.
\end{proof}
\begin{Corollary}\label{uni-chambers}
	Let $M$ be a higher graph manifolds as in definition \ref{HGM}. Let $C_{1},C_{2}$ be two chambers of $\til{M}$ and suppose that there exists an $r\geq 0$ such that $C_{1}\subset N_{r}(C_{2})$, then $C_{1}=C_{2}$.
\end{Corollary}
\begin{proof}
	Let $W_{1},W_{2}$ be distinct  walls, both of them adjacent to $C_{1}$. By Corollary \ref{wall-adjacent-to-chamber1}, they are adjacent to $C_{2}$. Thus $C_{1}=C_{2}$.
\end{proof}
\indent We now want to completely characterize the fundamental group of the pieces of a higher graph manifold $M$. We will focus on the action of $\pi_{1}(M)$ on $T$, and describe the set of stabilizers of walls of $\til{M}$.\\
\indent We say that a piece $Z$ of $M$ is a \textit{surface piece} if the base of the fiber bundle which defines $Z$ is a hyperbolic surface. \\
\indent In order to understand the stabilizers of walls, we define the following set:
	$$\mathcal{N}(\pi_{1}(M))=\{ H < \pi_{1}(M) \textrm{ } | \textrm{ }H \textrm{ a maximal nilpotent subgroup of } \pi_{1}(M) \textrm{ and }\hh(H)=n-1 \}$$
Here $\hh(H)$ is the Hirsch length of $H$. 
\begin{Proposition}\label{NG1}
	Let $M$ be a higher graph manifolds as in definition \ref{HGM} without surface pieces. Let $H$ be a subgroup of  $\pi_{1}(M)$. Then $H\in \mathcal{N}(\pi_{1}(M))$ if and only if, $H$ is a maximal nilpotent subgroup of the stabilizer of a vertex $v$ in the Bass-Serre tree $T$ associated to the decomposition of $M$ in pieces and has $\hh(H)=n-1$.
\end{Proposition}
\begin{proof}
	Let $Z_{v}$ be the piece corresponding to the vertex $v$ and let $N_{v}$ be the fiber as in definition \ref{HGM}. \\
	Suppose that $H\in \mathcal{N}(\pi_{1}(M))$, then by Proposition \ref{Nilpotent-groups-serre} either $H$ fixes a unique vertex, or there exists a path $\gamma$, fixed as a set under the action $H$, on which $H$ acts by translations. We will prove that $H$ can not act by translations on $\gamma$.\\
	By contradiction, assume $H$ acts by translations on the path $\gamma$. Denote by $\varphi:H\longrightarrow \mathbb{Z}$ the homomorphism through which $H$ acts by translations in $\gamma$ and $K=\Ker(\varphi)$. As $\hh(H)=n-1$ and  $\hh(\mathbb{Z})=1$, Lemma \ref{Hirsch-suma} implies $\hh(K)=n-2$. Moreover, $K$ is a subgroup of $H$ which fixes the path $\gamma$. \\
	For every point on $\gamma$, in particular for $v$, $K$ acts on the chamber $C_{v}$ which corresponds to $v$ and $K$ stabilizes two walls $W_{1},W_{2}$ of $C_{v}$ for which $\gamma$ enters and leaves. By Lemma \ref{estabilizador-paredes}, $K\leq \pi_{1}(N_{v})$. Hence the Hirsch length of $K$ is less than the Hirsch length of $N_{v}$, that is $\hh(K)<\hh(N_{v})$. By hypothesis we do not have surface pieces, then $\hh(N_{v})< n-3$, which is a contradiction. Thus $K$ only fixes a unique vertex of $\gamma$. Therefore $H$ is contained in the stabilizer of a vertex.\\
	\indent Suppose $H$ is a maximal nilpotent subgroup of the stabilizer of a vertex $v\in V(T)$ and is such that $\hh(H)=n-1$. Let $H'$ be a maximal nilpotent subgroup of $\pi_{1}(M)$, with $\hh(H')=n-1$ and containing $H$. There are two cases to consider.\\
	If $v$ is the unique vertex fixed by $H$, then as $H<H'$, $H'$ also fixes $v$. Therefore $H'$ is contained in the stabilizer of $v$ and by maximality $H=H'$.\\
	If $H$ fixes another vertex $w\neq v$, then $H$ must also fix an edge $e$ of $T$ exiting from $v$. As $H$ is a maximal subgroup of the stabilizer of $v$, $H$ coincides with the stabilizer of $e$. \\
	\indent In both cases we conclude that $H\in \mathcal{N}(\pi_{1}(M))$.
\end{proof}
 
\begin{Lemma}\label{NG2}
	Let $M$ be a higher graph manifold as in definition \ref{HGM}. If $H< \pi_{1}(M)$ is a wall stabilizer  then $H\in \mathcal{N}(\pi_{1}(M))$. On the other hand, if $H\in \mathcal{N}(\pi_{1}(M))$, then:
	\begin{enumerate}
		\item either $H$ is a wall stabilizer, or
		\item there exists a unique vertex $v\in V(T)$ fixed by $H$, and this vertex corresponds to a surface piece of $M$.
	\end{enumerate}
\end{Lemma}
\begin{proof}
	If $H$ is a wall stabilizer, then $H$ is a nilpotent maximal subgroup of the stabilizer of a vertex of $T$, hence Proposition \ref{NG1} implies $H\in \mathcal{N}(\pi_{1}(M))$.\\
	Now, suppose $H\in \mathcal{N}(\pi_{1}(M))$ is not a wall stabilizer.   By Proposition \ref{NG1}, $H$ is contained in the stabilizer of a vertex $v\in V(T)$. Moreover, $v$ is the unique vertex which is fixed by $H$, because otherwise $H$ would fix an edge and by maximality, $H$ will be a wall stabilizer. \\
	Let $Z_{v}$ be the piece corresponding to $v$ and $N_{v}^{k},B_{v}^{n-k}$ the fiber and base of $Z_{v}$, in that order. Suppose by contradiction that $Z_{v}$ is not a surface piece, hence $k\leq n-3$. Now, the Hirsch length of the projection of $H$ on $\pi_{1}(B_{v})$ is at least $n-k-1\geq 2$ and it is therefore contained in a cusp subgroup. By maximality, this implies that $H$ is a wall stabilizer, a contradiction.
\end{proof}
\begin{Lemma}\label{third-wall}
	Let $M$ be a higher graph manifold as in definition \ref{HGM}. Let $C$ be a chamber in $\til{M}$. Let $W_{1},W_{2}$ be two different walls adjacent to $C$ and let $H_{1},H_{2}$ be their stabilizers. Let $\overline{H}\in \mathcal{N}(\pi_{1}(M))\setminus \{H_{1},H_{2}\}$. Then, for each $k\geq 0$, there exist points $w_{1}\in W_{1}\cap C$, $w_{2}\in W_{2}\cap C$, which can be joined by a path $\gamma:[0,l]\ra C$ that does not cross the neighborhood of radius $k$, $N_{k}(\overline{W})$, of $\overline{W}$. Here $\overline{W}$ is the wall which is stabilized by $\overline{H}$. 
\end{Lemma}
The proof is analogous to the one of Lemma 7.8 in \cite{FLS}, here we only used  the wall stabilizer $H\in \mathcal{N}(\pi_{1}(M))$. We include the proof for completeness.
\begin{proof}
	As $\overline{H}\notin \{H_{1},H_{2}\}$ we can then assume that  $\overline{W}$ is disjoint from $C\cup W_{1} \cup W_{2}$. We thus need to consider the following two cases:
	\begin{enumerate}
		\item $\overline{W}$ lies in the connected component of $\widetilde{M}\backslash \{W_{1},W_{2}\}$ that contains $C$.
		\item $\overline{W}$ and $C$ lie in different connected components of $\widetilde{M}\backslash \{W_{1},W_{2}\}$.
	\end{enumerate}
	\textbf{Case 1:} If $\overline{W}$ lies in the connected component of $\widetilde{M}\backslash \{W_{1},W_{2}\}$ that contains $C$, then there exists a wall $W_{3}\neq W_{1}, W_{2}$ adjacent to $C$ such that every path that connects $\overline{W}$ with $W_{1}\cup W_{2}$ must pass through $W_{3}$.\\
	Every path that connects $W_{1}$ with $W_{2}$ but that does not pass through $N_{k}(W_{3})$ also avoids $\overline{W}$, because we have assumed that $\overline{W}$ is disjoint from $C\cup W_{1} \cup W_{2}$ and $W_{3}$ is adjacent to $C$. Let $\dd_{C}$ be the path metric on $C$. By Lemma \ref{distance-in-chambers}, it is enough to construct a path $\gamma$ that joins $w_{1}$ with $w_{2}$ and such that for all $t\in [0,l]$ and any given constant $k'>0$ we have $\dd_{C}(\gamma(t),W_{3})\geq k'$. \\ 
	Let $p:C\rightarrow B$ be the projection of chamber $C$ onto its base $B$. For $i=1,2,3$, let $\xi_{i}$ be points at infinity on $B$ and $O_{i}$ the horospheres centered at $\xi_{i}$ defined as follows: 
	\begin{enumerate}
		\item $O_{1}=p(W_{1}\cap C)$
		\item $O_{2}=p(W_{2}\cap C)$
		\item $O_{3}=p(W_{3}\cap C)$
	\end{enumerate}
	\begin{figure}[h]
		\centering
		\includegraphics[scale=0.25]{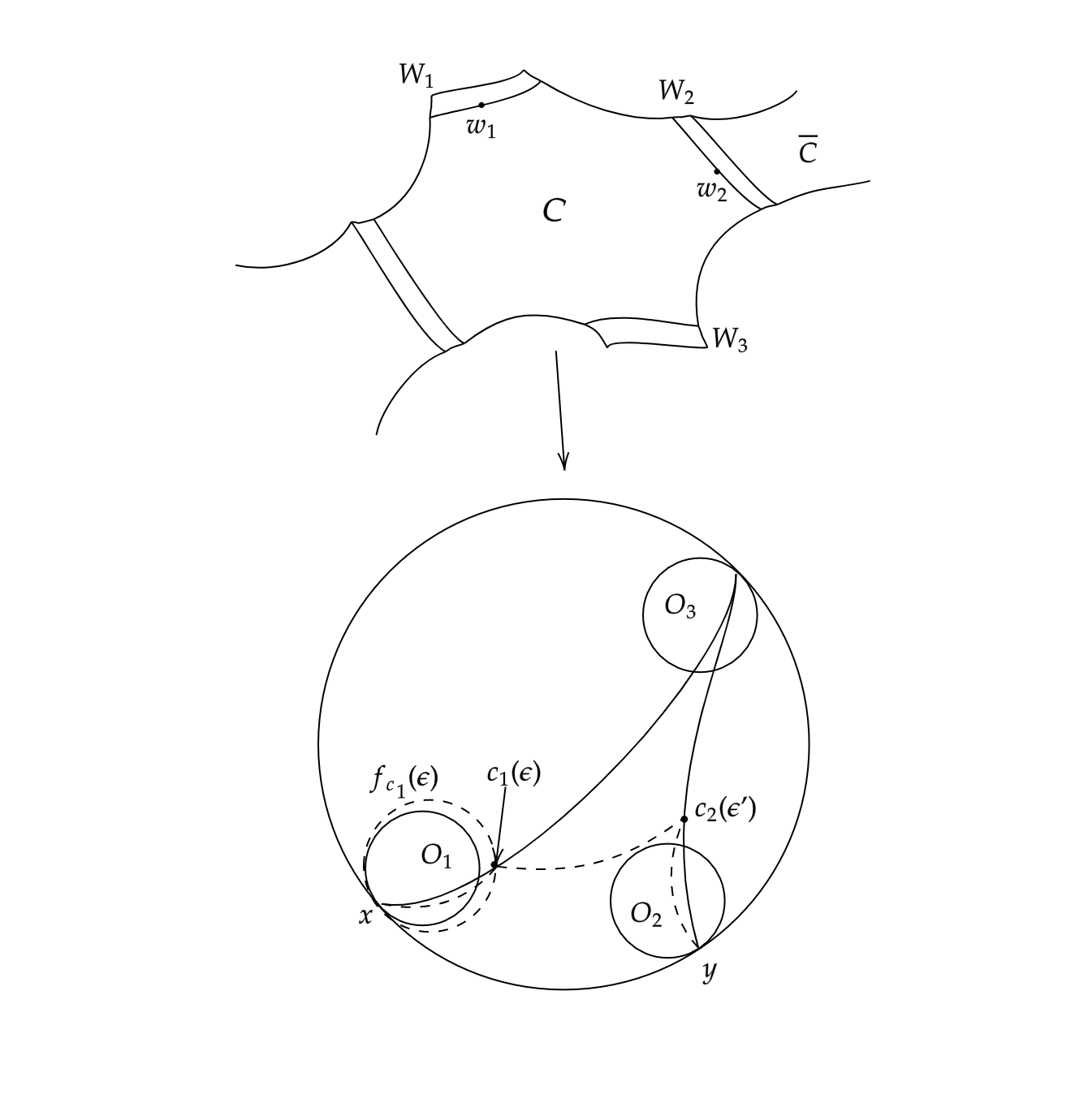}
		\caption{Construction of the projection of the path $\gamma$ between $w_{1}$ and $w_{2}$, from Lemma \ref{third-wall}.}
		\label{fig:path}
	\end{figure}
	Consider points $x\in O_{1}, y\in O_{2}$ and the geodesic rays $c_{1}:[0,\infty]\rightarrow B$ from $x$ to $\xi_{3}$ and $c_{2}:[0,\infty]\rightarrow B$ from $y$ to $\xi_{3}$. For every $t\in [0,\infty]$, we can consider a sequence of horospheres defined by the level sets of the Busemann function $f_{c_{i}}(t)$. For every sufficiently small $\epsilon>0$, we want to construct a path from any point on $O_{1}$ to any other point on $O_{2}$ and such that this path does not intersect $O_{3}$. Consider the concatenation of the following paths:
	\begin{enumerate}
		\item The subpath of $c_{1}$ from $x$ to the intersection point, $c_{1}(\epsilon)$, of one horospheres defined by $f_{c_{1}}(\epsilon)$ and $c_{1}$. 
		\item Consider the horosphere defined by the Bussemann function $f_{c_{3}}(t)$ that goes through the point $c_{1}(\epsilon)$ and call it $O_{c_{3}}$.
		\item The subpath of $c_{2}$ that goes from $y$ to the intersection point, $c_{2}(\epsilon')$, of $c_{2}$ and the horosphere of previous point. 
		\item The path that goes from $c_{1}(\epsilon)$ to $c_{2}(\epsilon')$ trough the horosphere $O_{c_{3}}$.
	\end{enumerate}  
	By construction, the concatenation of these three paths does not intersect $O_{3}$. Let $\gamma:[0,l]\rightarrow C$ the lift of this path to $C$. Therefore there exists a constant $K(\epsilon)$, which depends on $\epsilon$ and tends to $+\infty$ as $\epsilon$ goes to $0$, such that  for every $t\in [0,l]$,  $\dd_{C}(\gamma(t), \overline{W})\geq K(\epsilon)$. With this, we conclude the proof of the first case. \\
	\medskip 
	\textbf{Case 2:} Now, suppose that $\overline{W}$ and $C$ lie in different connected components of $\widetilde{M}\backslash \{W_{1},W_{2}\}$. Then we have that every path joining $W_{1}$ to $\overline{W}$ must pass through $W_{2}$ and that every path joining $W_{2}$ to $\overline{W}$ must pass through $W_{1}$.  Suppose that the first of these possibilities happens, by symmetry, the second case is analogous.\\
	Then, there exists a chamber $\overline{C}$ different to $C$ and adjacent to $W_{2}$. Choose a fiber $\overline{F}$ of $\overline{C} \cap W_{2}$ and let $F_{2}$ be the corresponding subspace of $W_{2}\cap C$. By Lemma \ref{distance-in-chambers}, if $k'$ is a given constant, then we want to construct a path $\gamma$ joining $W_{2}$ to $\overline{W}$ and such that $\dd_{C}(\gamma(t), \overline{F_{2}}) \geq k'$ for every $t\in [0,l]$. \\
	Let $B,F$ be the base and fiber of $C$, in that order. Then, we can choose $p\in F$ such that $B\times \{p\}\subseteq C$ intersects $F$ in a proper subspace of $W\cap (B\times \{p\})$. We know that $B$ can be identified with a negatively curved space without a finite number of non maximal horospheres removed, therefore we can identify $W_{2}\cap (B\times \{p\})$ with a horosphere $O_{2}$ centered at infinity and $W_{1}\cap (B\times \{p\})$ with a horosphere $O_{1}$, also centered at infinity. Using the same strategy of case 1, we can construct a path $\gamma$ as a concatenation of certain special subpaths in such a way that $\gamma$ and the neighbourhood of finite radius $\overline{O}$ of $F\cap (B\times\{p\})$ do not intersect each other. Finally, if the constant $k'$ is large enough, then any path in $B\times\{p\}$ that connects $W_{2}$ and $\overline{W}$ and that avoids the neighbourhood of radius $k'$ of $F\cap (B\times\{p\})$ also avoids the neighbourhood of radius $k'$ of $F$.
\end{proof}
\begin{Corollary}\label{Estab-isomorfismo}
	Let $M_{1},M_{2}$ be two higher graph manifolds as in definition \ref{HGM}. Let $\varphi: \pi_{1}(M_{1})\ra \pi_{1}(M_{2})$ be an isomorphism. If $H_{1}$ is a subgroup of $\pi_{1}(M_{1})$, then $H_{1}$ is the stabilizer of a wall in $\til{M}_{1}$ if and only if $\varphi(H_{1})$ is the stabilizer of a wall in $\til{M}_{2}$.
\end{Corollary}
\begin{proof}
	Suppose $H_{1}$ is the stabilizer of a wall in $\til{M}_{1}$, then $H_{1}$ is contained in the stabilizer of an edge $e$ of the Bass-Serre tree $T_{1}$ associated to $\til{M_{1}}$. As $\varphi$ is an isomorphism, $\varphi(H)$ is contained in the stabilizer of an edge $e_{2}$ of the Bass-Serre tree $T_{2}$ associated to $\til{M_{2}}$, i.e., it is contained in $gH_{2}g^{-1}$. We only need to show that $\varphi(H_{1})$ coincides with $gH_{2}g^{-1}$. The subgroup $\varphi^{-1}(gH_{2}g^{-1})$ of $\pi_{1}(M_{1})$ contains $H_{1}$ and again it is contained in the stabilizer of an edge of the Bass-Serre tree $T_{1}$. We conclude, $\varphi^{-1}(gH_{2}g^{-1})=H_{1}$.
\end{proof}
\begin{Proposition}(\cite[Proposition 4.13]{FLS})\label{qi-cubrientes}
	Let $M_{1}, M_{2}$ be two higher graph manifolds as in definition \ref{HGM}. Let $f:\til{M}_{1}\ra \til{M}_{2}$ be a $(k,c)-$quasi-isometry and let $g$ be its quasi-inverse. Suppose that there exists a $\lambda$ with the property that for each wall $W_{1}$ of $\til{M}_{1}$, there exists a wall $W_{2}$ of $\til{M}_{2}$ with Hausdorff distance between $f(W_{1})$ and $W_{2}$ bounded by $\lambda$, and that if we change $W_{1}$ with $W_{2}$ we have the same for $g$. \\
	Then, there exists a universal constant $H$, with the property that for each chamber $C_{1}$ of $\til{M}_{1}$, there exists a unique chamber $C_{2}$ in $\til{M}_{2}$ such that the Hausdorff distance between $f(C_{1})$ and $C_{2}$ is bounded by $H$.
\end{Proposition}
We include the proof for completeness, following the one presented by \cite{FLS}, which also applies to $M$.
\begin{proof}
	Let $W_{1},W_{1}'$ be walls adjacent to a fixed chamber $C_{1}\subseteq \widetilde{M}_{1}$. By hypothesis, there exist two walls $W_{2},W_{2}'$ in $\widetilde{M_{2}}$ such that $f(W_{1})$ and $f(W_{1}')$ are at finite Hausdorff distance of $W_{2}$ and $W_{2}'$ respectively, and by Lemma \ref{wall1-subset-nbh-wall2} these walls are unique.\\
	We want to prove that there exists a chamber $C_{2}\subseteq\widetilde{M_{2}}$ such that $W_{2}$ and $W_{2}'$ are adjacent walls of $C_{2}$.\\
	Suppose that there exists a wall $P_{2}\subseteq\widetilde{M_{2}}$ different to  $W_{2}$ and $W_{2}'$ such that every path which connects $W_{2}$ with $W_{2}'$ crosses $P_{2}$. Then, there exists a wall $P_{1}\subseteq\widetilde{M_{1}}$ such that $f(P_{1})$ is at most at finite distance $\lambda$ of $P_{2}$. 
	Since $f$ and $g$ are quasi-isometries, if $P_{2}$ separates $W_{2}$ and $W_{2}'$, then there exists a constant $D>0$ such that every path which connects $W_{1}$ with $W_{1}'$ intersects $N_{D}(P_{1})$. However, by Lemma \ref{third-wall} this is a contradiction. Therefore $W_{2}$ and $W_{2}'$ are adjacent walls to $C_{2}$.\\
	Let $d_{i}$ be the diameter of a chamber $Z_{i}$ and let $h$ be the maximum of $d_{i}/2$. Then for each $p_{1}\in C_{1}$, there exists $p_{1}'\in W_{1}$ with $d(p_{1},p_{1}')\leq h$, here $W_{1}$ is a adjacent wall to $C_{1}$. Thus
	\begin{equation*}
		\begin{array}{rcl}
		d(f(p_{1}),C_{2})& \leq &d(f(p_{1}),f(p_{1}'))+d(f(p_{1}),C_{2})\\
		& \leq & kh+c+\lambda.
		\end{array}
	\end{equation*}
	From which we obtain that $f(C_{1})$ is in the neighbourhood of radius $kh+c+\lambda$ of $C_{2}$. In a similar way, we can see that $g(C_{2})$ is in the neighbourhood of radius $kh+c+\lambda$ of some chamber $C_{1}'$, but by Corollary \ref{uni-chambers} $C_{1}=C_{1}'$. \\
	Since $g$ is the quasi-inverse of $f$, if $q_{2}\in C_{2}$, then $d(q_{2},f(g(q_{2})))\geq c$ and also, there exists $q_{1}\in C_{1}$ with $d(g(q_{2}),q_{1})\geq kh+c+\lambda$. We now want to estimate the distance between each element $q_{2}\in C_{2}$ and $f(q_{1})$, where $q_{1}\in C_{1}$:
	\begin{equation*}
		\begin{array}{rcl}
		d(q_{2},f(q_{1})) & \leq & d(q_{2},f(g(q_{2})))+d(f(g(q_{2})),f(q_{1}))\\
		& \leq & c+kd(g(q_{2}),q_{1})+c\\
		& \leq & 2c+k(kh+c+\lambda)
		\end{array}
	\end{equation*}
	With $L=2c+k(kh+c+\lambda)$ we obtain the result. Finally, the uniqueness of $C_{2}$ is a consequence of Lemma \ref{uni-chambers}.
\end{proof}
We are now ready to prove Theorem \ref{Isomorphism-preserve-pieces}.
\begin{proof}[Proof of Theorem \ref{Isomorphism-preserve-pieces}]
	Let $\Lambda_{1}<\pi_{1}(M_{1})$ be the fundamental group of a piece $Z_{1}$ of $M_{1}$. As a consequence of Proposition \ref{qi-cubrientes} and the Milnor-Svar$\check{\textrm{c}}$ Lemma (Proposition \ref{lemaMS}), the Hausdorff distance $\dd_{H}(\varphi(\Lambda_{1}), g\Lambda_{2}g^{-1})$ is bounded above by $L$ for some $\Lambda_{2}<\pi_{1}(M_{2})$ which is the fundamental group of a piece in $M_{2}$ and some $g\in\pi_{1}(M_{2})$. \\
	Without loss of generality, we may assume that $g=id$.\\
	If $h\in\Lambda_{1}$, we obtain that: \[\varphi(h)\varphi(\Lambda_{1})=\varphi(h\Lambda_{1})=\varphi(\Lambda_{1})\]\\
	Since the Hausdorff distance $\dd_{H}(\varphi(\Lambda_{1}),\Lambda_{2})$ is bounded above, then $\varphi(h)\Lambda_{2}$ is at bounded Hausdorff distance of $\Lambda_{2}$. By Milnor-Svar$\check{\textrm{c}}$ Lemma, if $C_{2}$ is a chamber fixed by $\Lambda_{2}$, then $\dd_{H}(\varphi(h)C_{2}, C_{2})$ is finite. This yields $\varphi(h)C_{2}=C_{2}$ and $\varphi(h)\in\Lambda_{2}$, and therefore $\varphi(\Lambda_{1})\subseteq \Lambda_{2}$.\\
	Using the quasi-inverse $\varphi^{-1}$ of $\varphi$ we can prove that $\varphi^{-1}(\Lambda_{2})\subseteq\Lambda_{1}$.\\
	By Proposition \ref{qi-cubrientes}, for $h\in \Lambda_{1}$ the Hausdorff distance $\dd_{H}(\varphi^{-1}(\Lambda_{2}),h\Lambda_{1}h^{-1}) < H$. \\
	Again, we can assume that $h=id$, so for $g\in\Lambda_{2}$ we have that
		$$\varphi^{-1}(g)\varphi^{-1}(\Lambda_{2})=\varphi^{-1}(g\Lambda_{2})=\varphi^{-1}(\Lambda_{2}).$$
	Then, as $\dd_{H}(\varphi^{-1}(\Lambda_{2}),h\Lambda_{1}h^{-1})<H$, we find that $\dd_{H}(\varphi^{-1}(g)\cdot \Lambda_{1},\Lambda_{1})$ is bounded. \\
	If $\Lambda_{1}$ fixes a chamber $C_{1}\in \widetilde{M}_{1}$, then $\dd_{H}(\varphi^{-1}(g)(C_{1}), C_{1})< \infty$. 
	Therefore, by Corollary \ref{uni-chambers} $\varphi^{-1}(g)(C_{1})=C_{1}$, so $\varphi^{-1}(g)\in \Lambda_{1}$ and $\varphi^{-1}(\Lambda_{2})\subseteq \Lambda_{1}$. \\
	\indent We conclude that $\varphi(\Lambda_{1})=\Lambda_{2}$.
\end{proof}

\section{Cusp-Decomposable manifolds}\label{section4}
	
In this section we will prove Theorem \ref{quasi-isometric-embeding} for cusp-decomposable manifolds. \\
\indent Frigerio, Lafont and Sisto mention that except for Proposition 7.8 all the results of their section VII \cite{FLS} can be modified to hold for cusp-decomposable manifolds. They mention that a proof of Proposition 7.8 for cusp-decomposable manifolds can not be similar to the one that they presented, (Remark 7.9 \cite{FLS}). Armed with knowledge of the electric space associated to a pinched Hadamard manifold, we will prove the corresponding result to Proposition 7.8 for the case of cusp decomposable manifolds in Proposition \ref{good-path}.\\
\indent Tam Nguy$\tilde{\hat{\textrm{e}}}$n-Phan described the following family of graph manifolds that are a subfamily of the manifolds in Definition \ref{HGM} where the pieces are all pure pieces \cite{Tam}. \\
\indent Let $V$ be a locally symmetric, complete, finite volume, noncompact, connected, negatively curved manifold of dimension $n\geq 3$. It is know that $V$ has a finite number of cusps and each cusp is diffeomorphic to $S\times [0,\infty]$, where $S$ is a compact $(n-1)-$dimension manifold \cite{Eberlein}. Let $b$ be large enough so that the boundary components $S\times \{b\}$ of different cusps do not intersect each other. Now delete $S\times (b,\infty)$ from each cusp, then the resulting space $Z$ is a compact manifold with boundary. The lifts of the boundaries components of $Z$ are horospheres in $\widetilde{Z}$. This manifold is called a \textbf{bounded cusp manifold with horoboundary}.
\begin{Definition}(Tam Nguyen Phan,\cite{Tam})\label{CDM}
	A \textbf{\textit{cusp decomposable manifold} $M$} is a manifold which is obtained by taking a finite number of bounded cusp manifolds with horoboundary $Z_{i}$ and identifying their horoboundaries using affine diffeomorphisms. 
\end{Definition}
Throughout this section $M$ will denote a cusp-decomposable manifold.\\
\indent Let $M$ be a cusp-decomposable manifold of dimension $n$. The universal cover $\widetilde{M}$ has a natural structure as a tree of spaces as follows. A \emph{chamber} $C\in \widetilde{M}$ will be the preimage of a connected component in $\widetilde{M}$ of $Z$ without the deleted cusp $S\times (b,\infty)$. A wall $W\in\widetilde{M}$ will be the preimage of a connected component of $S\times[0,b]$. \\
\indent For each wall $W$ of $\til{M}$ its boundary can be decomposed into two connected components  $W_{+}$ and $W_{-}$ that are the intersection of a chamber with one adjacent wall. We will call each of the two components $W_{+}, W_{-}$ the \textit{thin walls} associated to $W$. Denote by $\dd_{\pm}$ the path metric of $W_{\pm}$, induced by the restriction of the Riemannian structure of $\til{M}$. For each wall $W$, let $s_{W}:W_{+}\ra W_{-}$ be a map that sends $p\in W_{+}$ to the point $s_{W}(p)\in W_{-}$ which is joined to $p$. We think of $s_{W}$ as the map that glues the two pieces, $W_{+}, W_{-}$, together.
\begin{Lemma}\label{walls-embedding-in-chambers}
	Let $M$ be a cusp-decomposable manifold as in definition \ref{CDM} and let $W$ be a wall in $\til{M}$. If $C$ is a chamber that contains $W_{+}$, then the inclusion of $(W_{+},\dd_{W_{+}})\hookrightarrow (C,\dd_{C})$  is an isometry.
\end{Lemma}
\begin{proof}
	Recall that if $Z$ is the piece of the manifold $M$ associated to the chamber $C$, then every component of $\partial C$ is a convex horosphere in the metric sense (see \cite{Farb}).\\
	As a consequence, if $W_{+}$ is a thin wall contained in $C$, then $W_{+}$ is a convex hypersurface of $C$. Therefore, the path metric induced in $W_{+}$ by the Riemannian structure on $\til{M}$ is isometric to the restriction of $\dd_{C}$.
\end{proof}
\begin{Proposition}\label{thin-walss-bi-Lipschitz-embedding} Let $M$ be a cusp-decomposable manifold as in definition \ref{CDM}. Let $W$ be a wall in $\widetilde{M}$ and $W_{\pm}$ the thin walls associated to $W$. Then, the inclusion $(W_{\pm},\dd_{{W}_\pm})\hookrightarrow (W,\dd_{W})$ is a bi-Lipschitz embedding and a quasi-isometry.
\end{Proposition}
\begin{proof}
	The inclusion $i:W_{\pm}\hookrightarrow W$ is 1-Lipschitz, by definition of the induced path metric. This map induces an isomorphism on fundamental groups, so by the Milnor-Svar$\check{\textrm{c}}$ Lemma $i$ is a quasi-isometry. Therefore $i$ is a bi-Lipschitz embedding at large scale, i.e. there exist constants $b\geq 0$, $c\geq 1$ such that if $\dd_{W_{\pm}}(x,y)\geq b$ then $\dd_{W_{\pm}}(x,y)\leq c \cdot \dd_{W}(x,y)$. \\
	\indent We now only need to analyze the case where $0\leq \dd_{W_{\pm}}(x,y)\leq b$. Let $T$ be the subset of  $W_{\pm}\times W_{\pm}$ defined by the inequality $0\leq \dd_{W_{\pm}}(x,y)\leq b$. Let $Z$ be a piece of $M$ and let $S$ be the infranilmanifold of its boundary. Then, $\til{S}$ is homeomorphic to $W_{+}$ and $W_{-}$. Consider the following action of $\pi_{1}(S)$ on $T$: 
	\begin{equation*}
		\begin{array}{rcl}
		\rho: \pi_{1}(S)\times T & \rightarrow & T\\
		   (g, (p,s_{W}(p))) & \mapsto & (gp, s_{W}(gp))
		\end{array}
	\end{equation*}
	Therefore, $T$ is invariant under this action and the quotient space $T/\pi_{1}(S)$ is compact.\\
	Define a function $f: T \rightarrow \mathbb{R}$ as follows:
	 \[ f(x,y) = \left\{
	 	\begin{array}{ll}
	 	 1 & \textit{for every pair } (p,s_{W}(p))\in (W_{+},W_{-}),\\
	 	 \dd_{{W}_\pm}(x,y)/\dd_{W}(x,y) & \textit{if } (x,y)\in T\backslash \{(p,s_{W}(p))| p\in W_{+} \textit{ and } s_{W}(p)\in W_{-}\}
	 	\end{array}
	 \right.
	 \]
	This is a positive continuous function and the compactness of $T/\pi_{1}(S)$ implies that it is bounded above by some constant $c'$. If $C=\max \{c,c'\}$, then the inclusion map $i$ is C-bi-Lipschitz and we conclude the proof. 
\end{proof}
Let $W_{\pm}$ be a thin wall and $x,y\in W_{\pm}$. Let $\gamma$ be a path which connects $x$ with $y$ in $\til{M}$. We will say that $\gamma$ is a \textit{non-backtracking path} in $W_{\pm}$ if $\gamma$ only intersects the wall  $W_{\pm}$ in its endpoints. \\
\indent The following proposition is a consequence of Lemma 3.2 of \cite{Osin} and the proof is the same as that of the Proposition 7.4 of \cite{FLS}.
\begin{Proposition}\label{paths-horospheres1}
	Let $Z$ be a bounded cusp manifold with horoboundary. Then there exists a constant $\lambda$ that depends only on $Z$ such that the following is true. Let $\gamma\subseteq Z$ be a loop obtained by concatenating a finite number of paths $\alpha_{1},\gamma_{1},...,\alpha_{n},\gamma_{n}$, where
	\begin{itemize}
		\item Each $\alpha_{i}$ is a geodesic in the horosphere $S_{i}\subseteq \partial Z$.
		\item Each $\gamma_{i}$ is a path in $Z$ that connects the final point of $\alpha_{i}$ with the initial point of $\alpha_{i+1}$.
		\item The final points of each $\gamma_{i}$ are in different walls.
	\end{itemize}
	Let $\Lambda\subseteq \{1,...,n\}$ be a subset of indices such that $S_{k}\neq S_{i}$ for each $k\in \Lambda$, $i\in \{1,...n\}$, $i\neq k$. Then,
	\begin{equation*}
		\sum_{k\in\Lambda} L(\alpha_{k}) \leq \lambda\cdot \sum_{i=1}^{n}L(\gamma_{i}).
	\end{equation*}
\end{Proposition}
We say that $\gamma$ is \textit{minimal}, if for each chamber $C$ the set $\gamma \cap \mathring{C}$ is only a finite collection of paths and each of these paths connects different walls of $C$. Moreover, we say that $\gamma$ is \textit{good} if it is minimal and for each thin wall $W_{\pm}$ contained in $C$, there are at most two final points of paths in $\gamma \cap \mathring{C} $ that belong to $W_{\pm}$. Another characterization of good paths is as follows. Let $W_{\pm}$ be a thin wall contained in a wall $W$. A path  $\gamma:[t_{0},t_{1}]\rightarrow \widetilde{M}$ is \textit{external} to $W_{\pm}$ if $\gamma(t_{0}),\gamma(t_{1})\in W_{\pm}$ and $\gamma|_{(t_{0},t_{1})}$ is supported on $\widetilde{M}\backslash W$.  Let $\gamma$ be a minimal path and $n$ be the number of external subpaths of $\gamma$ to $W_{\pm}$, the \emph{exceeding number} of $\gamma$ on $W_{\pm}$ is defined as $\max \{0,n-1\}$. The exceeding number $e(\gamma)$ of $\gamma$ is defined as the sum over all the thin walls of the exceeding number of $\gamma$. Denote by $j(\gamma)$ the sum over all the chambers of $\widetilde{M}$ of the number of connected components of $\gamma\cap \mathring{C}$. A path $\gamma$ is good if its minimal and $e(\gamma)=0$. 
\begin{Lemma}\label{auxiliar-goodpath}
	 Let $M$ be a cusp-decomposable manifold as in definition \ref{CDM}. Let $x,y\in W\subset \widetilde{M}$ be two points in the same wall and $\beta\geq 1$. Let $D$ be the electric constant of Lemma \ref{electric4.8}. If $\dd(x,y)\leq D$ then there exists a good path $\gamma\in \widetilde{M}$ from $x$ to $y$ such that $L(\gamma)\leq \beta \dd(x,y)$.
\end{Lemma} 
\begin{proof}
	Consider the set of pairs $(x,y)\in W$ such that $\dd(x,y)\leq D$. This set is a compact set. Now, pick good paths $\gamma_{xy}$ for each such pair. We have that the map $(x,y)\mapsto L(\gamma_{xy})$ is continuous and compact, therefore it is bounded above by a constant. \\
	So for each pair $x,y$ there exists a constant $\beta_{xy}$ such that 
	$L(\gamma_{xy})\leq \beta_{xy} \dd(x,y)$. Set $\beta$ equal to the maximum over all $\beta_{xy}$ to obtain the result.
\end{proof}
\begin{Proposition} \label{good-path}
	  Let $M$ be a cusp-decomposable manifold as in definition \ref{CDM}. Let $x,y$ be points in the same wall of $\widetilde{M}$. Then there exists a constant $\beta \geq 1$ that depends only on the geometry of $\widetilde{M}$, such that the following is true. There exists a good path $\gamma$ on $\widetilde{M}$ that connects $x$ with $y$ and such that $L(\gamma)\leq \beta \cdot \dd(x,y)$. 
\end{Proposition}
\begin{proof}
	\begin{figure}[h]
		\centering
		\includegraphics[scale=0.25]{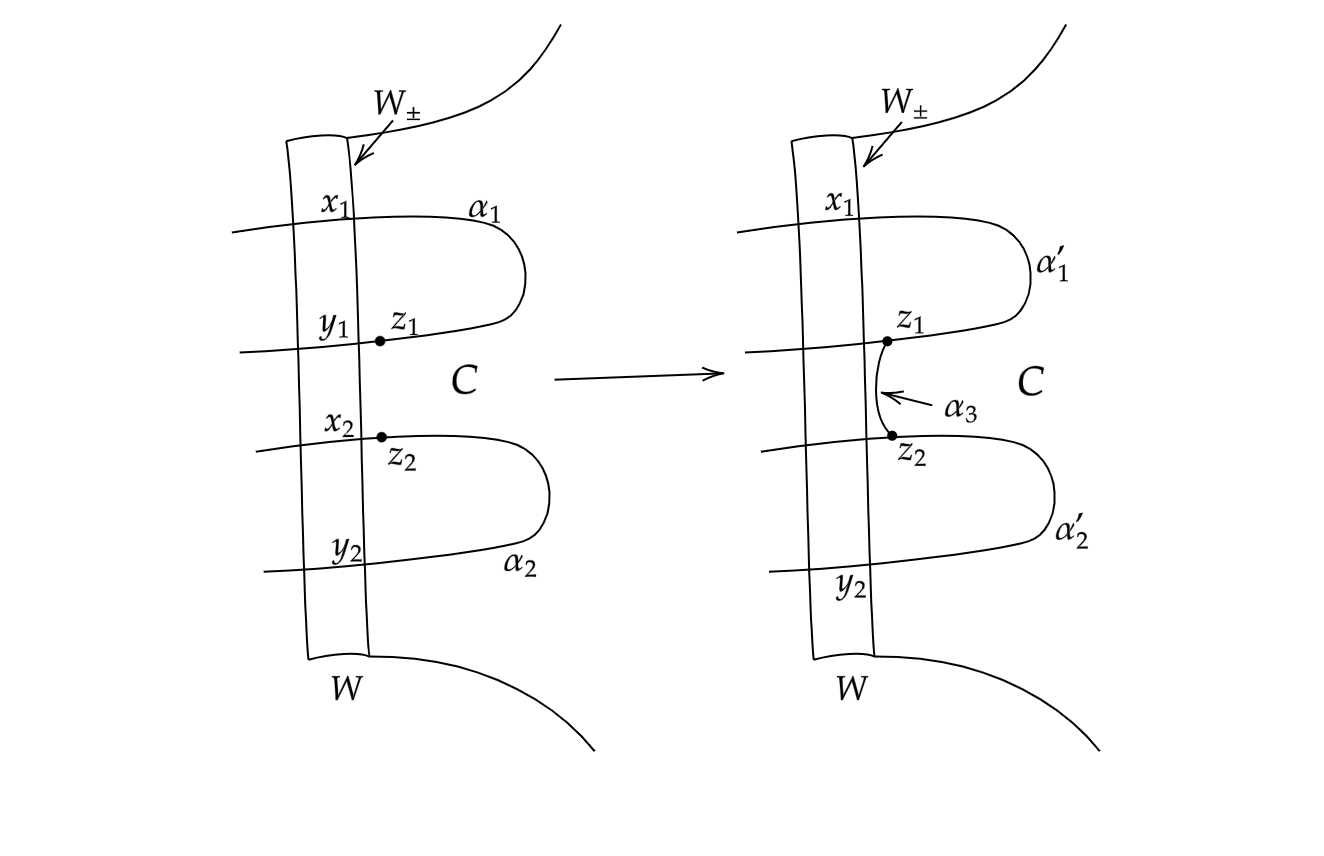}
		\caption{Construction of a good path interpolating between $x_1$ and $y_2$, that goes through the points $z_1$ and $z_2$, from Proposition \ref{good-path}.}
		\label{fig:good-path}
	\end{figure} 
	Let $\wh{M}$ be the electric space associated to $M$. Remember that $R$ is the distance between any two of the deleted horospheres of $V$. 
	By Lemma \ref{Lemma4.4Farb}, we know that the horospheres are visually bounded. Let $\Delta>0$ be a constant such that every horospherical boundary component of each $Z_{i}$ is visually bounded by $\Delta/2$.\\
	\indent Let $\gamma_{0}$ be a geodesic joining $x$ and $y$ in $\widetilde{M}$. As every thin wall is convex, then every geodesic in $\widetilde{M}$ is minimal, therefore $\gamma_{0}$ is minimal.\\
	\indent We want to modify a minimal path $\gamma_{i}$ with $e(\gamma_{i})>0$ by a new minimal path $\gamma_{i+1}$ such that $j(\gamma_{i+1})<j(\gamma_{i})$ and $L(\gamma_{i+1})\leq L(\gamma_{i})+K$, where $K$ is a positive constant. As $j(\gamma_{0})\leq L(\gamma_{0})/R=\dd(x,y)/R$ then after at most $\dd(x,y)/R$ steps, we will obtain a new minimal path $\gamma$ such that $e(\gamma)=0$.\\
	\indent The reader might find it useful to see Figure \ref{fig:good-path} while following the next construction. Suppose that we have some external subpaths $\alpha_{1}=[x_{1},y_{1}]$ and $\alpha_{2}=[x_{2},y_{2}]$ of $\gamma_{i}$, contained in the interior of $C$, and such that $x_{1},y_{1},x_{2},y_{2}\in W_{\pm}$ where $W_{\pm}$ is a thin wall of a chamber $C$. Choose two points $z_{1},z_{2}\in \gamma_{i}\cap \mathring{C}$ in such a way that $L([z_{1},y_{1}])< 1/2$ and $L([y_{2},z_{2}])< 1/2$. The new path $\gamma_{i+1}'$ is constructed as follows. Start by taking the subpath $\alpha_{1}'=[x_{1},z_{1}]$, then consider a path $\alpha_{3}$ in $\mathring{C}$ from $z_{1}$ to $z_{2}$ but such that $L([z_{1},z_{2}])<\Delta +1$ and the subpath $\alpha_{2}'=[z_{2},y_{2}]$. \\
	\indent The path $\gamma_{i+1}'$ satisfies $j(\gamma_{i+1}')=j(\gamma_{i})-1$ and its external number equals zero. So, we have a good path between $x_{1}$ and $y_{2}$. By Lemma \ref{auxiliar-goodpath}, there exists a constant $\beta(\Delta/2)$, such that $L(\gamma_{i+1}')\leq \beta(\Delta/2)\cdot \dd(x_{1},y_{2})$.\\
	\indent We need to carry out these replacements for all the external subpaths of $\gamma_{i}$. After performing all these replacements, we end with a good  path that satisfies the length condition.
\end{proof}
\begin{Lemma}\label{chambers-then-wals-bilipschitz}
	 Let $M$ be a cusp-decomposable manifold as in definition \ref{CDM}. Let $W\subseteq \widetilde{M}$ be a fixed wall, $x,y$ points in $W_{\pm}$, and $C$ the chamber that contains $x,y$. Suppose that there exists $\alpha\geq 1$ such that if $x,y$ can be joined by a path $\gamma$ in $\widetilde{M}$ which does not backtrack in $W_{\pm}$, then 
	\begin{equation*}
		\dd_{C}(x,y)\leq \alpha \cdot L(\gamma).
	\end{equation*}
	Therefore the inclusion of $W$ into $\widetilde{M}$ is a bi-Lipschitz embedding.
\end{Lemma}
\begin{proof}
	As the inclusion $(W,\dd_{W})\hookrightarrow \widetilde{M}$ is 1-Lipschitz, we only need to prove that there exists a constant $k\geq 1$ such that for all $p,q\in W$ we have that $\dd_{W}(p,q)\leq k\cdot \dd(p,q)$. By Proposition \ref{good-path}, there exists a good path $\gamma$ joining $p$ and $q$ such that for $\beta\geq 1$, $L(\gamma) \leq \beta\dd(p,q)$. \\
	Let $m$ be the number of chambers adjacent to $W$ whose interior intersects $\gamma$. \\
	Let us relabel the following subpaths of $\gamma$ as follows:
	\begin{enumerate}
		\item For $i=1,...,m$, let $\gamma_{i}^{W}$ be a path contained in $W$.
		\item For $i=1,...,m$, let $\gamma_{i}^{W_{\pm}}$ be a good path with endpoints on the thin walls $W_{\pm}$ and which does not backtrack in $W_{\pm}$.
	\end{enumerate}
	With this notation, we can decompose $\gamma$ in the following form:
	$$\gamma=\gamma_{1}^{W}\gamma_{1}^{W_{\pm}}...\gamma_{m}^{W}\gamma_{m}^{W_{\pm}}\gamma_{m+1}^{W}.$$
	Let $C$ be a chamber adjacent to $W$. We want to replace the paths $\gamma_{i}^{W_{\pm}}$ by paths $\eta_{i}$ in $C$ and such that the total length of the curve obtained does not exceed $\alpha\cdot L(\gamma)$.\\
	Let $W_{1},W_{2}$ be two walls different from $W$ and adjacent to $C$. As $\gamma_{i}^{W_{\pm}}$ is a good path, then there exist subpaths of $\gamma_{i}^{W_{\pm}}$ that connect different walls and these subpaths only cross the walls in two points. \\
	Figure \ref{fig:gamma-by-eta} illustrates the following construction. Let $p',q'$ be the endpoints of the good path $\gamma_{1}^{W_{\pm}}$. Suppose that $\gamma_{1,1}^{W_\pm}$ is the subpath of $\gamma_{1}^{W_{\pm}}$ that connects $W$ with $W_{1}$ and that $\gamma_{1,2}^{W_\pm}$ is the subpath of $\gamma_{2}^{W_{\pm}}$ that connects $W_{1}$ with $W_{2}$. Let $p_{1}$ be the point of $W_{1,\pm}$ where $\gamma_{1,1}^{W_\pm}$ crosses $W_{1,\pm}$ and let $q_{1}$ be the point of $W_{1,\pm}$ where $\gamma_{1,2}^{W_\pm}$ crosses $W_{1,\pm}$. Let $\eta_{p_{1},q_{1}}^{C}$ be the geodesic path inside $C$ between $p_{1}$ and $q_{1}$. With this process, we are cutting the path inside $W_{1}$ and replacing it by a path contained in $C$ and such that its length is less than $k$ times the length of the path inside $W_{1}$, for some positive constant $k$. Now, repeat this process for all the other subpaths $\gamma_{1,i}^{W_{\pm}}$ of $\gamma_{1}^{W_{\pm}}$ between the rest of the walls adjacent to $C$. Let $\eta_{1}$ be the path formed as follows. Start in $p'$ and follow the subpath of $\gamma_{1,1}^{W_\pm}$ between $p'$ and $p_{1}$, then on $p_{1}$ switch to following the path $\eta_{p_{1},q_{1}}^{C}$ and repeat this procedure until returning to $q'$. The path $\eta_{1}$ between $p'$ and $q'$ is a path completely contained in $C$ and by hypothesis its length is bounded by $\alpha_{1}\cdot L(\gamma_{1}^{W_{\pm}})$.
	\begin{figure}[h]
		\centering
		\includegraphics[scale=0.22]{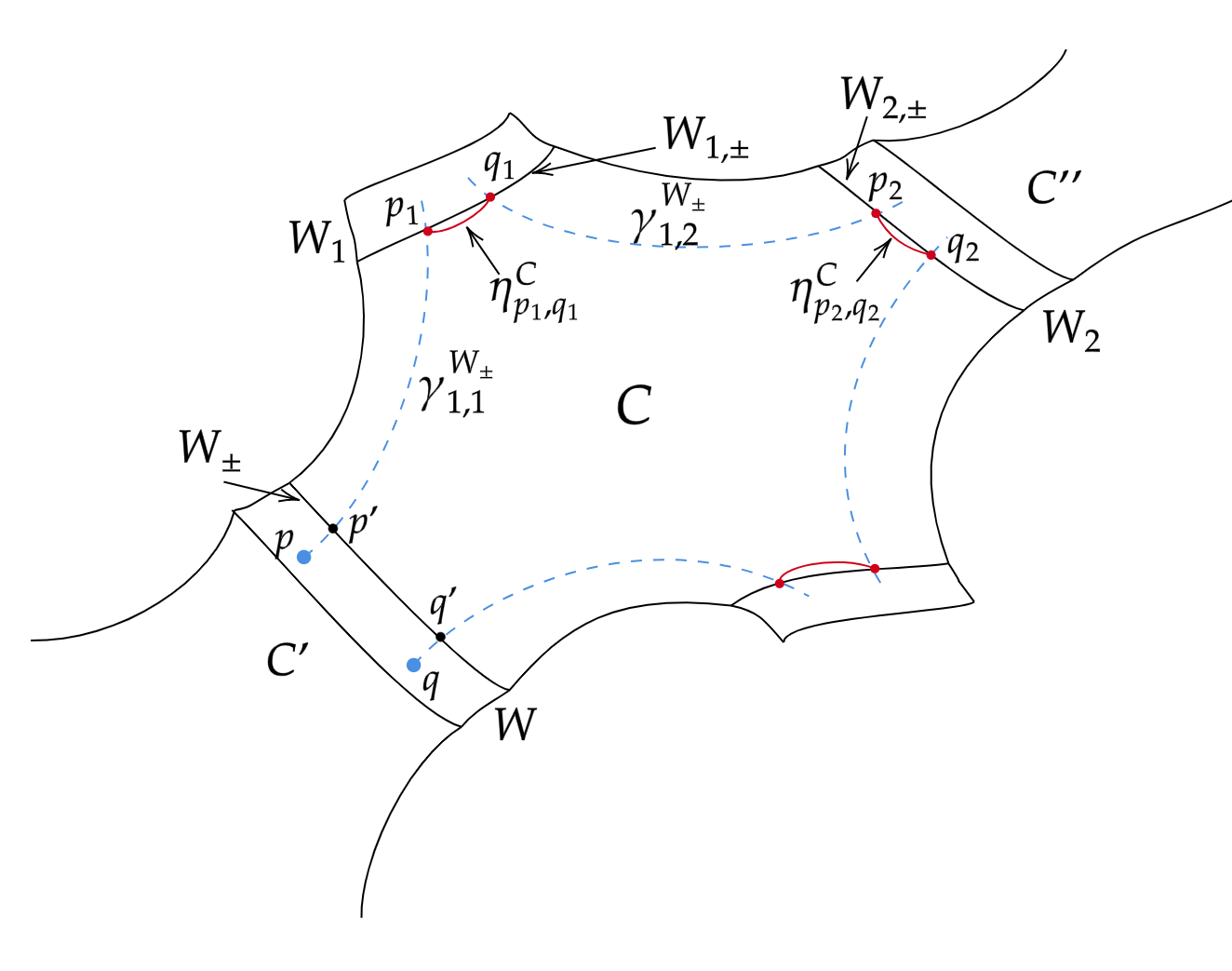}
		\caption{Here we illustrate how the paths $\eta_{p_{i},q_{i}}^{C}$ replace $\gamma_{i}^{W}$ by paths inside $C$. With the same notation as in Lemma \ref{chambers-then-wals-bilipschitz}, the points $p',q'$ lie on $W_{\pm}$ and the points $p_{i},q_{i}$ lie on the thin walls $W_{i,\pm}$. The paths $\gamma_{i,j}^{\pm}$ are good paths joining the thin walls $W_{i,\pm}, W_{j,\pm}$ and passing through $p_{i+1},q_{i}$.}
		\label{fig:gamma-by-eta}
	\end{figure} 
	Repeat this for all the paths $\gamma_{i}^{W_{\pm}}$ in $\gamma$. After replacing all the curves, the resultant curve $\eta=\gamma_{1}^{W}\eta_{1}...\gamma_{m}^{W}\eta_{m}\gamma_{m+1}^{W}$ has length less that $\alpha \cdot L(\gamma)$. Moreover, we have that $$\dd_{W}(p,q)\leq L(\eta) \leq \alpha \cdot L(\gamma)$$
	If $ k = \alpha\cdot  \beta$ then we have shown that $\dd_{W}(p,q) \leq k \cdot \dd(p,q)$.
\end{proof}

As a consequence of the previous results, we are now able to prove that the walls and chambers embeddings in $\til{M}$ are bi-Lipschitz.
		
\begin{Proposition}\label{walls-bilipschitz}
	 Let $M$ be a cusp-decomposable manifold as in definition \ref{CDM} and let $W\subseteq \widetilde{M}$ be a wall. Then the inclusion of $(W,\dd_{W})\hookrightarrow \widetilde{M}$ is a bi-Lipschitz embedding. In particular, it is a quasi-isometric embedding. Moreover, the bi-Lipschitz constant of the embedding only depends on the geometry of $\widetilde{M}$. 
\end{Proposition}
\begin{proof}
	Let $W_{+}$ be a thin wall of $W$ and $C$ be the chamber which contains $W_{+}$. Let $p,q$ be two points on $W_{+}$. Proposition \ref{good-path} guarantees that there exists a good path $\gamma$ in $\widetilde{M}$ that joins $p,q$ and which does not backtrack on $W_{+}$. By Lemma \ref{chambers-then-wals-bilipschitz}, we only need to prove that there exists a constant $\alpha\geq 1$ only depending on the geometry of $\widetilde{M}$ such that the following condition it is true:
	\begin{equation*}
		\dd_{C}(p,q)\leq \alpha\cdot L(\gamma).
	\end{equation*}
	Observe that as $\dd_{C}(p,q)$ is the path metric over $C$, by definition this is the infimum over all the paths that joins $p,q$ therefore with $\alpha=1$ we obtain the result.
\end{proof}
\begin{Proposition}\label{chambers-bilipschitz}
	Let $M$ be cusp decomposable manifold, then the inclusion of a chamber in $\widetilde{M}$ is a bi-Lipschitz embedding.
\end{Proposition}
\begin{proof}
	Let $p,q$ be points in a chamber $C\in \widetilde{M}$ and $\gamma$ a geodesic from $p$ to $q$. We can decompose $\gamma$ as follows. Let $\gamma_{i}$ be geodesics in $C$ and let $p_{i},q_{i}\in W_{i,+}$ be the endpoints of a path $\eta_{i}$ where $W_{i,+}$ is a thin wall adjacent to $C$. So we may write  $\gamma=\gamma_{1}\eta_{1}...\gamma_{n}\eta_{n}\gamma_{n+1}$. \\
	By Lemma \ref{chambers-then-wals-bilipschitz}, there exists a constant $a \geq 1$ such that $\dd_{W_{i}}(p_{i},q_{i})\leq a\cdot \dd(p_{i},q_{i})$. Then $\dd_{W_{i,+}}(p_{i},q_{i})\leq a\cdot \dd(p_{i},q_{i})$. We can replace each $\eta_{i}$ with a path $\eta_{i}'\in W_{i,+}$  that has the same endpoints as $\eta_{i}$ but such that its length is less than $a \dd(p_{i},q_{i})$. So, the new path $\gamma'=\gamma_{1}\eta_{1}'...\gamma_{n}\eta_{n}'\gamma_{n+1}$ is contained in $C$ and has length at most $ a' \dd(p,q)$. Therefore $\dd_{C}(p,q)\leq a' \dd(p,q)$. 
\end{proof}
We can now present the main result of this section.
\begin{Theorem}\label{qi-embedded}
	Let $M$ be a cusp-decomposable manifold. Then, the inclusion of chambers and walls (with their path metric) in $\widetilde{M}$ are quasi-isometric embeddings.  
\end{Theorem}
\begin{proof}
	Lemma \ref{walls-embedding-in-chambers} implies that $C,W\hookrightarrow \widetilde{M}$ are isometric embeddings. Lemma \ref{walls-bilipschitz} and Proposition \ref{chambers-bilipschitz} imply that $C,W\hookrightarrow \widetilde{M}$ are bi-Lipschitz embeddings. Therefore chambers and walls are quasi-isometrically embedded in $\widetilde{M}$. 
\end{proof}

\end{document}